\pdfoutput=1
\documentclass[a4paper]{scrartcl}
\usepackage[utf8]{inputenc}
\usepackage[T1]{fontenc}
\usepackage[english]{babel}
\usepackage{graphicx}
\usepackage{float}
\usepackage{amsmath,amssymb,amsthm,amscd,stmaryrd}
\usepackage{twoopt}
\usepackage{verbatim} 
\usepackage[pdftitle={Coupled skinny baker's maps and the Kaplan-Yorke conjecture},
pdfauthor={Maik Gröger and Brian Hunt}, bookmarks=true, pdfborder={0 0 0},
plainpages=false, pdfpagelabels]{hyperref}

\def\clap#1{\hbox to 0pt {\hss#1\hss}}
\def\mathclap{\mathpalette\mathclapinternal}
\def\mathclapinternal#1#2{%
	\clap{$\mathsurround=0pt#1{#2}$}%
}

\newtheorem{theorem}{Theorem}[section]
\newtheorem{corollary}[theorem]{Corollary}
\newtheorem{proposition}[theorem]{Proposition}
\newtheorem{lemma}[theorem]{Lemma}
\newtheorem{conjecture}[theorem]{Conjecture}
\theoremstyle{definition}
\newtheorem{definition}[theorem]{Definition}

\newcommand{\set}[1]{\mathbb #1}

\newcommand{\abs}[1]{\left|#1\right|}
\newcommand{\Abs}[1]{\left\|#1\right\|}
\newcommand{\supnorm}[1]{\left\|#1\right\|_{\infty}}
\newcommand{\Lip}[1]{\mathrm{Lip}(#1)}

\newcommand{\Cone}{C_b^1(M)}
\newcommand{\balg}[1]{\mathcal B\big(#1\big)}

\newcommand{\Buc}{B}
\newcommand{\muuc}{\mu}
\newcommand{\Auc}{A}
\newcommand{\Iuc}{I}

\newcommand{\Bc}{B_g}
\newcommand{\muc}{{\mu_g}}
\newcommand{\Ac}{A_g}
\newcommand{\mucae}[1]{$\muc$\hspace{1mm}-\hspace{1mm}a.e. #1}
\newcommand{\hc}[1][]{{h_g^{#1}}}
\newcommand{\Ic}{I_g}

\newcommand{\Ip}{I_p}

\newcommand{\mup}{\muc_{\!\scriptscriptstyle\lambda}}
\newcommand{\hcp}[1][]{{h_g}_{\!\scriptscriptstyle\lambda}^{#1}}
\newcommand{\gp}{g_\lambda}

\title{Coupled skinny baker's maps and the Kaplan-Yorke conjecture}
\author{
Maik Gröger\thanks{groeger@math.uni-bremen.de; Fachbereich 3 - Mathematik,
Universität Bremen, Bibliothekstraße 1, 28359 Bremen, Germany}
\and
Brian R. Hunt\thanks{bhunt@umd.edu; Department of Mathematics and Institute for
Physical Science and Technology, University of Maryland, College Park MD 20742,
USA}}
\date{28. February 2013}

\begin{document}
\renewcommand\dagger{**}
\maketitle
\let\thefootnote\relax
\begin{abstract}
	The Kaplan-Yorke conjecture states that for ``typical'' dynamical systems
	with a physical measure, the information dimension and the Lyapunov dimension
	coincide.
	We explore this conjecture in a neighborhood of a system for which the two
	dimensions do not coincide because the system consists of two uncoupled
	subsystems.
	We are interested in whether coupling ``typically'' restores the equality of
	the dimensions. 
	The particular subsystems we consider are skinny baker's maps, and we
	consider uni-directional coupling.  For coupling in one of the possible
	directions, we prove that the dimensions coincide for a prevalent set of
	coupling functions, but for coupling in the other direction we show that the
	dimensions remain unequal for all coupling functions.
	We conjecture that the dimensions prevalently coincide for bi-directional
	coupling.
	On the other hand, we conjecture that the phenomenon we observe for a
	particular class of systems with uni-directional coupling, where the
	information and Lyapunov dimensions differ robustly, occurs more generally
	for many classes of uni-directionally coupled systems (also called
	skew-product systems) in higher dimensions.
\end{abstract}

\section {Introduction}

In 1979, Kaplan and Yorke introduced a quantity that is known today as the
Lyapunov or Kaplan-Yorke dimension, denoted by $D_L$ in the following. 
It is defined in terms of the Lyapunov exponents of a differentiable map.
In \cite{KaplanYorke1979} they conjectured that for an attractor of a 
``typical'' dynamical system in a Euclidean space, the box-counting dimension of
the attractor is equal to  $D_L$.
In \cite{FarmerOttYorke1983} and \cite{FredericksonKaplanYorkeYorke1983} the
conjecture was refined to replace the box-counting dimension with several
suitably defined dimensions of the physical measure on the attractor; part of
the conjecture is that ``typically'' these definitions are well-defined and all
coincide with the Lyapunov dimension of this measure.
Nowadays, this conjecture is called the Kaplan-Yorke conjecture.
The most commonly used dimension included in the conjecture is the
information dimension, which we will also use here and denote by $D_1$.
We call the conjectured equation $D_1=D_L$ the Kaplan-Yorke equality.

A physical measure is an invariant probability measure that is ``observable''
from a positive Lebesgue measure set of initial conditions, and thus should be
(approximately) observable in an experiment or in a computer simulation.
An attractor can have many different invariant measures, each with its own
Lyapunov exponents and information dimension, and for arbitrary invariant
measures $D_1$ and $D_L$ generally do not coincide.
For example, the invariant measure supported on an unstable periodic orbit has
$D_1 = 0$ but $D_L>0$.
Existence of physical measures is widely assumed based on empirical
evidence, but known rigorously only in limited cases.
We will consider the Kaplan-Yorke conjecture to apply only to systems for which
physical measures exist.
Common examples of physical measures are so-called SRB measures, see
\cite{Young2002}.
Some authors use the term SRB measure for physical measure, see also 
\cite{Young2002} for a short discussion.
We are following the terminology of this reference.

Under the assumption that the Kaplan-Yorke equality holds, $D_1$ can be
obtained by calculating $D_L$, which is determined by dynamical quantities,
namely, the Lyapunov exponents.
Generally, numerically estimating $D_L$ is easier than estimating $D_1$ directly,
especially in higher dimensions.
By estimating $D_1$, we get geometric measure theoretic information about the
complexity of the attractor.
Further, we quantify the amount of information that is necessary
to specify the state of a system to a certain accuracy.

The Kaplan-Yorke conjecture is broad in the sense that it does not specify a
precise class of systems to which it should apply, and it does not specify what
exactly ``typical'' means.
The conjecture has been proved for some specific classes of systems, for
example, in \cite{Young1982} for surface diffeomorphisms with an SRB measure,
in \cite{LedrappierYoung1988} for compositions of random diffeomorphisms, and in
\cite{AlexanderYorke1984} or \cite{KaplanMallet-ParetYorke1984} for certain
systems depending on a finite number of parameters.

A simple class of systems where the Kaplan-Yorke equality typically does not hold
consists of systems that can be decomposed into two or more uncoupled subsystems.
Indeed, the starting point for this article is such a system.
It consists of two 2-dimensional skinny baker's maps and is given by
\begin{align*}
	x_{n+1}&=2x_n\mod 1\notag\\
	y_{n+1}&=
	\begin{cases}
		\alpha y_n&\textnormal{ if }0\leq x_n<\frac{1}{2}\\
		\alpha y_n+1-\alpha&\textnormal{ if }\frac{1}{2}\leq x_n<1,
	\end{cases}\notag\\
	z_{n+1}&=2z_n\mod 1\\
	w_{n+1}&=
	\begin{cases}
		\beta w_n&\textnormal{ if }0\leq z_n<\frac{1}{2}\\
		\beta w_n+1-\beta&\textnormal{ if }\frac{1}{2}\leq z_n< 1
	\end{cases}\notag
\end{align*}
with $0<\alpha,\beta<\frac{1}{2}$.
For this uncoupled system, we have that the information dimension $D_1$ is
strictly less than the Lyapunov dimension $D_L$, except when $\alpha=\beta$.
That means the conjectured equality does not hold for Lebesgue almost every
$(\alpha,\beta)\in\big(0,\frac{1}{2}\big)^2$.
The reason for this is that $D_1$ is additive but not $D_L$, in the sense of
getting the dimension of the uncoupled system by adding the dimensions of the
subsystems.

Now, the question arises whether we can find a larger class of dynamical systems
that contains the systems described above but where the Kaplan-Yorke equality is 
typically valid.
Because of the independent behavior of the subsystems, coupling seems to be
the natural way to find this larger class of dynamical systems.
In our case we choose the following form of coupling
\begin{align}
	\label{coupled_skinny_bakers_maps}
	x_{n+1}&=2x_n\mod 1\notag\\
	y_{n+1}&=
	\begin{cases}
		\alpha y_n+f(z_n,w_n)&\textnormal{ if }0\leq x_n<\frac{1}{2}\\
		\alpha y_n+1-\alpha+f(z_n,w_n)&\textnormal{ if }\frac{1}{2}\leq x_n<1,
	\end{cases}\notag\\
	z_{n+1}&=2z_n\mod 1\\
	w_{n+1}&=
	\begin{cases}
		\beta w_n+g(x_n,y_n)&\textnormal{ if }0\leq z_n<\frac{1}{2}\\
		\beta w_n+1-\beta+g(x_n,y_n)&\textnormal{ if }\frac{1}{2}\leq z_n< 1
	\end{cases}\notag
\end{align}
where $0<\alpha,\beta<\frac{1}{2}$ and where $g$ and $f$ are $C_b^1$.
The uncoupled system corresponds to $f=g=0$.	
More generally, we could add coupling terms to the $x$- and $z$-equations, but we
think the class of systems \eqref{coupled_skinny_bakers_maps} is sufficiently
rich.
We consider the Kaplan-Yorke conjecture for $f$ and $g$ that are typical in
the sense of prevalence \cite{HuntSauerYorke1992}, described in Section 2.
Our results are for the case where one of these two functions is zero, though as
we discuss below, we think this gives a good indication of the situation when
both are nonzero.
Interestingly, our results depend on which coupling function is nonzero (and the
relative size of $\alpha$ and $\beta$).
They suggest that the usual definition of Lyapunov dimension is not always
appropriate in the case of uni-directionally coupled (skew-product) systems.

Our main result is as follows.
\begin{theorem}\label{main_theorem}
	For the system \eqref{coupled_skinny_bakers_maps}, if $f=0$ and $g$ is
	$C_b^1$, then for
	\begin{enumerate}
		\item[(i)] $\alpha>\beta$: $D_1<D_L$ for all $g$,
		\item[(ii)] $\alpha=\beta$: $D_1=D_L$ for all $g$,
		\item[(iii)] $\alpha<\beta$: $D_1=D_L$ for a prevalent set of $g$'s.
	\end{enumerate}
\end{theorem}

Since the problem is symmetric for uni-directional coupling, we get the following
as a direct implication.

\begin{corollary}
	If $g=0$ and $f$ is $C_b^1$, then for 
	\begin{enumerate}
		\item[(i)] $\alpha>\beta$: $D_1=D_L$ for a prevalent set of $f$'s,
		\item[(ii)] $\alpha=\beta$: $D_1=D_L$ for all $f$,
		\item[(iii)] $\alpha<\beta$: $D_1<D_L$ for all $f$.
	\end{enumerate}
\end{corollary}
Notice that for each pair $(\alpha,\beta)$, at least one type of uni-directional
coupling can change the attractor enough to make $D_1=D_L$.
As a result, we conjecture that in the case of bi-directional coupling,
$D_1=D_L$ for prevalent $f$ and $g$.

\begin{figure}[H]
	\includegraphics[width=1.0\textwidth]{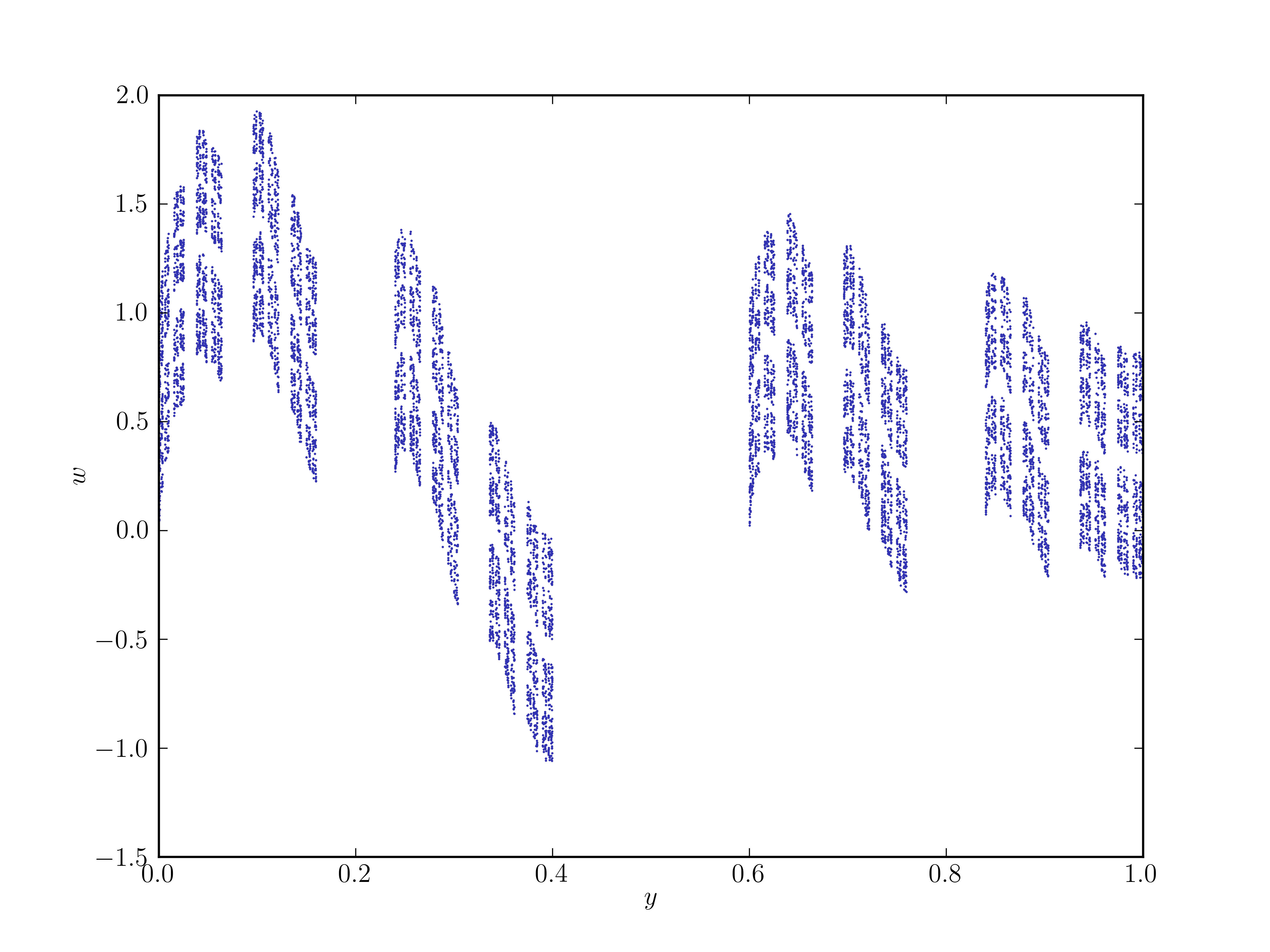}
	\caption{Plot of the intersection of the 4-dimensional attractor of the 
		uni-directionally coupled system for $\alpha=0.4$, $\beta=0.43$ and
		$g(x,y)=\cos\big(\frac{\pi}{2}x\big)\sin\big(\frac{3\pi}{2}y\big)$
		with the $w$-$y$ plane for fixed $x$ and $z$.
		Different values of $x$ and $z$ give similar cross-sections.
		Topologically, the attractor is the product of a Cantor dust-like set,
		as shown in the figure above, and the unit square.}
	\label{sample_attractor}
\end{figure}

With Figure \ref{sample_attractor} the reader can get an impression of what the
attractor of the uni-directional coupled system can look like.
The main task will be to determine the dimension of the physical measure that is 
supported on this attractor.

The outline of this article is as follows.
In the next section, we recall the main notions that are necessary to state the 
properties of the coupled system.
We also restate the Kaplan-Yorke conjecture more precisely, adjusted to our setting.
In the third section, we will prove all the assertions (in particular, Corollary
\ref{corollary_D_L} and Theorem \ref{prevelant_result}) that are needed in order
verify Theorem \ref{main_theorem}.
We conclude the article with a discussion of our results, their implications for
more general systems, and related work on the dimensionality of filtered chaotic
signals.

\bigskip

\noindent{\bf Acknowledgments.} 
The majority of this work was carried out while the first author was undertaking
a Fulbright scholarship at the University of Maryland, College Park.
Also, the first author was partially funded by an Emmy-Noether-grant of the
German Research Council (DFG-grant JA 1721/2-1).
The authors are grateful to Daniel Karrasch and Markus Waurick for many helpful
and enlightening discussions.
This work is related to the activities of the Scientific Network ``Skew product
dynamics and multifractal analysis'' (DFG grant OE 538/3-1).

\section{Preliminaries}

In the following, we will denote by $\mathcal B(M)$ the Borel $\sigma$-algebra 
of a subset $M\subseteq\set R^d$ and with $\lambda^d$ the Lebesgue measure on
$\set R^d$.
Further, $\mathring M$ will denote the interior of $M$ and $\overline M$ the
closure of $M$.

\begin{definition}
	Let $F:M\subseteq\set R^d\to M$ be a map with $M$ locally compact and 
	assume there exist finitely many pairwise disjoint connected subsets 
	$U_i\in\balg M$ such that 
	\[
		M=\bigcup\limits_i U_i
	\]
	and	the map $\left.F\right|_{U_i}$ is continuous for each $U_i$ (with respect
	to the relative topology) plus $\left.F\right|_{\mathring{U_i}}$ is $C^1$.
	Furthermore, we assume that
	\[
		\max_i\Abs{\left. (F\right|_{\mathring{U_i}})'}_\infty<\infty
	\]
	and	$\lambda^d(M_0)=0$ with
	\[
		M_0:=\bigcup\limits_{n\in\set N_0}
			F^{-n}\left(M\backslash\bigcup\limits_i\mathring{U_i}\right).
	\]
	Then, we call $F$ a \emph{piecewise $C^1$ dynamical system}.
\end{definition}

Note that a subset $M\subseteq\set R^d$ is locally compact if and only if there
exist an open subset $A\subseteq\set R^d$ and a closed subset $B\subseteq\set R^d$
such that $M=A\cap B$ \cite[Theorem 6.5]{Dugundji1966}.
This means in particular $M\in\mathcal B(\set R^d)$ and therefore 
$M_0\in\mathcal B(\set R^d)$.
We require $\lambda^d(M_0)=0$ to avoid that orbits could be mapped into an open
subset where the derivative is not defined.
Such an open subset could be considered as a hole and therefore could cause
positive escape rates which in turn would involve a different definition of the
Lyapunov dimension.

\begin{definition}
	Let $F:M\subseteq\set R^d\to M$ be a piecewise $C^1$ dynamical system and
	let $\mu$ be an $F$-invariant Borel probability measure on $M$.
	We call $\mu$ a \emph{physical measure} if there exists a set $V\subseteq M$
	of positive Lebesgue measure such that for every bounded continuous function 
	$\varphi:M\to\set R$,
	\begin{align}
		\label{property_physical_measure}
		\lim\limits_{N\to\infty}\frac{1}{N}\sum\limits_{n=0}^{N-1}\varphi(F^n(x))
		=\int\varphi d\mu,
	\end{align}
	for every $x\in V$.
\end{definition}

Now, we will define the Lyapunov dimension following \cite{AlexanderYorke1984}.

\begin{definition}
	Let $F:M\subseteq\set R^d\to M$ be a piecewise $C^1$ dynamical system.
	Assume that $F$ has an ergodic invariant measure $\mu$ with
	$\mu(M_0)=0$. 
	Let 
	\[
		\chi_1(\mu)\geq\chi_2(\mu)\geq\dots\geq\chi_d(\mu)
	\]
	be the Lyapunov exponents and set
	\[	
		j:=\max\{i:\chi_1(\mu)+\dots+\chi_i(\mu)\geq 0\},
	\]
	(respectively, $0$ if $\chi_1(\mu)<0$).
	We define the \emph{Lyapunov} or \emph{Kaplan-Yorke dimension} as
	\begin{align*}
		D_L(\mu):=
		\begin{cases}
			0&\textnormal{ if }j=0\\
			j+\frac{\chi_1(\mu)+\dots+\chi_j(\mu)}{\abs{\chi_{j+1}(\mu)}}
				&\textnormal{ if }1\leq j<d\\
			d&\textnormal{ if }j=d
		\end{cases}.
	\end{align*}			
\end{definition}

Note that by the required assumptions in the last definition the existence of the
Lyapunov exponents is ensured by Oseledets theorem.
The Kaplan-Yorke conjecture is referring to this situation, see
\cite{FarmerOttYorke1983} or \cite{FredericksonKaplanYorkeYorke1983}.
In these references, one can also find a heuristic derivation and interpretation
of the Lyapunov dimension.
For more information on the Lyapunov exponents, see e.g.\ \cite{BarreiraPesin2002}.

Before we can state the Kaplan-Yorke conjecture for our setting, we will define
all the other dimensions that we need in this paper, following mainly 
\cite{HuntKaloshin1997}.
From now on, let $\mu$ be a Borel probability measure supported on a bounded
subset $A\subset\set R^d$ and let $B(x,\varepsilon)\subset\set R^d$ be an open 
ball around $x\in\set R^d$ with radius $\varepsilon>0$.

\begin{definition}
	The \emph{lower} and \emph{upper box-counting dimension} of $A$ are defined as
	\begin{align*}
		\underline D_B(A):=\liminf\limits_{\varepsilon\to 0}
			\frac{\log N(A,\varepsilon)}{-\log\varepsilon},\\
		\overline D_B(A):=\limsup\limits_{\varepsilon\to 0}
			\frac{\log N(A,\varepsilon)}{-\log\varepsilon},
	\end{align*}
	where $N(A,\varepsilon)$ is the smallest number of balls of radius
	$\varepsilon$ that can cover $A$.
	If $\underline D_B(A)=\overline D_B(A)$, then their common value $D_B(A)$ is
	called the \emph{box-counting dimension} of $A$.
\end{definition}

\begin{definition}
	For each point $x$, we define the \emph{lower} and
	\emph{upper pointwise dimension} of $\mu$ at $x$ to be
	\begin{align*}
		\underline d(\mu,x):=\liminf\limits_{\varepsilon\to 0}
			\frac{\log\mu(B(x,\varepsilon))}{\log\varepsilon},\\
		\overline d(\mu,x):=\limsup\limits_{\varepsilon\to 0}
			\frac{\log\mu(B(x,\varepsilon))}{\log\varepsilon}.
	\end{align*}
	If $\underline d(\mu,x)=\overline d(\mu,x)$, then their common value 
	$d(\mu,x)$ is called the \emph{pointwise dimension} of $\mu$ at $x$.
	We say that the measure $\mu$ is \emph{exact dimensional} if the pointwise
	dimension exists and is constant almost everywhere, i.e.\
	\[
		\underline d(\mu,x)=\overline d(\mu,x)=:d(\mu),
	\]
	for $\mu$-almost every $x\in A$.
\end{definition}

\begin{proposition}[{\cite[Proposition 10.3]{Falconer1997},
	\cite[Proposition 3.8]{Falconer2003}}]
	\label{proposition_relation_pointwise_boxcounting_dimension}
	For $\mu$-almost every $x\in A$ we have
	\[
		\overline d(\mu,x)\leq\overline D_B(A).
	\]
\end{proposition}

\begin{definition}
	The \emph{lower} and \emph{upper information dimension} of $\mu$ are defined
	as
	\begin{align*}
		\underline D_1(\mu):=\liminf\limits_{\varepsilon\to 0}
			\frac{\int\log\mu(B(x,\varepsilon))d\mu(x)}{\log\varepsilon},\\
		\overline D_1(\mu):=\limsup\limits_{\varepsilon\to 0}
			\frac{\int\log\mu(B(x,\varepsilon))d\mu(x)}{\log\varepsilon}.
	\end{align*}
	If $\underline D_1(\mu)=\overline D_1(\mu)$, then their common value 
	$D_1(\mu)$ is called the \emph{information dimension} of $\mu$.
\end{definition}

Commonly, a grid based definition of the information dimension is used but for
our setting the integral and grid based definition coincide, see
\cite[Theorem 2.1]{BarbarouxGerminetTcheremchantsev2001}.

\begin{theorem}[{\cite[Theorem 2.2]{Cutler1991},\cite[Theorem 1]{MyjakRudnicki2007}}]
	\label{theorem_relation_pointwise_information_dimension}
	We have
	\[
		\int\underline d(\mu,x)d\mu(x)\leq\underline D_1(\mu)
		\leq\overline D_1(\mu)\leq\int\overline d(\mu,x)d\mu(x).
	\]
\end{theorem}

That means, provided the pointwise dimension exists almost everywhere,
\[
	D_1(\mu)=\int d(\mu,x)d\mu(x),
\]
i.e.\ we can interpret the information dimension of the measure $\mu$ as 
the averaged pointwise dimension of $\mu$.
In particular, if $\mu$ is exact dimensional, then $D_1(\mu)=d(\mu)$.
Note that in this case also several other dimensions of the measure coincide
\cite{Young1982}.

Now, we can state the Kaplan-Yorke conjecture for our setting.

\begin{conjecture}
	Given a locally compact subset $M\subseteq\set R^d$. 
	For ``typical'' piecewise $C^1$ dynamical systems $F:M\to M$ with an ergodic
	invariant physical measure $\mu$, we have that
	\[
		D_1(\mu)=D_L(\mu).
	\]
\end{conjecture}

In \cite{FredericksonKaplanYorkeYorke1983} it is conjectured that in this context
$\mu$ is exact dimensional with $d(\mu)=D_L(\mu)$ and therefore $D_1(\mu)=D_L(\mu)$.
Indeed, our main result will be of this type.

As already mentioned in the introduction, the word ``typical'' is not precisely
defined in the conjecture.
Thereby, one problem is that in infinite dimensional vector spaces there is no
natural notion of typical phenomena, in the sense of ``Lebesgue almost everywhere'',
respectively, ``Lebesgue measure zero''.
One way to define it is to use the topological notion based on the
category theorem of Baire.
Prevalence is another concept to provide an analog of what typical could mean
in the context of infinite dimensional vector spaces.
In our case, we are dealing with the space of $C^1$ maps where the map and the
derivative are bounded.
We refer to \cite{OttYorke2005} and \cite{HuntKaloshin2010} for more general
definitions and examples regarding the notion of prevalence.

\begin{definition}
	Let $V$ be a completely metrizable topological vector space.
	A Borel measure $\nu$ is said to be \emph{transverse} to a Borel set
	$E'\subset V$ if there exists a compact subset $S\subset V$ with
	$0<\nu(S)<\infty$ and $\nu(E'+v)=0$ for all $v\in V$.
	A subset $E\subset V$ will be called \emph{shy} if there exist a Borel set
	$E'\subset V$ with $E\subseteq E'$ and a measure $\nu$ that is transverse to
	$E'\subset V$.
	The complement of a shy set is called a \emph{prevalent} set.
\end{definition}

Note that if this concept is applied to $V=\set R^d$ the only transverse 
measure	is Lebesgue measure.
This motivates the following definition.

\begin{definition}
	A finite dimensional subspace $P\subset V$ will be called a probe for a set
	$T\subset V$ if Lebesgue measure supported on $P$ is transverse to a Borel set
	containing the complement of $T$.
\end{definition}

One of our main tools will be the potential theoretic method for the pointwise
dimension.

\begin{definition}
	For $s\geq 0$ define the \emph{$s$-potential} of $\mu$ at a point 
	$x\in\set R^d$ as
	\[
		\varphi_s(\mu,x):=\int\frac{1}{\abs{x-y}^s}d\mu(y).
	\]
\end{definition}

\begin{theorem}[\cite{SauerYorke1997}]
	\label{pointwise_dimension_equal_sup_s_potential}
	For $x\in A$ we have
	\[
		\underline d(\mu,x)=\sup\{s:\varphi_s(\mu,x)<\infty\}
		=\inf\{s:\varphi_s(\mu,x)=\infty\}.
	\]
\end{theorem}

\section{Proofs}

From now on, we assume
\[
	0<\alpha,\beta<\frac{1}{2}
\]
and set 
\[
	M:=[0,1)\times\set R.
\]

First, we define the 2-dimensional and the uncoupled skinny baker's map.

\begin{definition}
	The \emph{(2-dimensional) skinny baker's map} is defined as 
	\begin{align*}
		B_\alpha:M\subset\set R^2\to M:(x,y)\mapsto
		\begin{cases}
			(2x,\alpha y)&\textnormal{ if }0\leq x<\frac{1}{2}\\\\
			(2x-1,\alpha y+1-\alpha)&\textnormal{ if }\frac{1}{2}\leq x<1
		\end{cases}.
	\end{align*}
	The \emph{uncoupled skinny baker's map} is defined as
	\[
		\Buc:M^2\subset\set R^4\to M^2:(x,y,z,w)\mapsto(B_\alpha(x,y),B_\beta(z,w)),
	\]
	where $M^2=M\times M$.
\end{definition}

We will state some basic facts  of the 2-dimensional skinny baker's map, for
more details see \cite{AlexanderYorke1984} and \cite{FarmerOttYorke1983}.
The attractor of $B_{\alpha}$ is just the product of the interval $[0,1)$ 
in the $x$-direction and a Cantor set (determined by the parameter $\alpha$ and
denoted by $A_\alpha$ in the following) in the $y$-direction.
The box-counting dimension of the attractor is $1-\frac{\log 2}{\log\alpha}$
(this also holds for the Hausdorff dimension of the attractor).
The physical measure $\mu_\alpha$ of $B_{\alpha}$ is unique (the basin is
Lebesgue almost every point) and it is the product of the Lebesgue measure in
the $x$-direction and the Cantor measure in the $y$-direction, denoted by
$\nu_\alpha$ in the following.
Furthermore, $\mu_\alpha$ is strong-mixing and is exact dimensional with the 
same value as the box-counting dimension.

From these facts we can deduce some properties of the uncoupled skinny baker's
map.
The product set $\Auc:=[0,1)\times A_\alpha\times[0,1)\times A_\beta$ is
invariant under $\Buc$ and we will see in the first proposition that $\Auc$ is
also the attractor of $\Buc$.
Further, the product measure $\muuc:=\mu_\alpha\times\mu_\beta$ is invariant
under $\Buc$ and is also strong-mixing \cite[Proposition 1.6]{Brown1976}, i.e.\
in particular ergodic.
Later, we will see that it is the unique physical measure of the uncoupled
system, too.
Since the box-counting and Hausdorff dimension of the attractor of $B_\alpha$
coincide and since the pointwise dimension is additive, we have
\[
	D_B(\Auc)=d(\muuc)=2-\frac{\log 2}{\log\alpha}-\frac{\log 2}{\log\beta}.
\]

Next, we will define the coupled skinny baker's maps.
In order to do this, we will need a space of coupling functions.
For that purpose, consider an open subset $U\subseteq\set R^d$ and denote by
$C_b^1(U)$ the space of all $C^1$ maps $g:U\to\set R$ where $g$ and $g'$ are
bounded.
Note that $C_b^1(U)$ equipped with the norm 
$\Abs{g}_{1,\infty}:=\max\{\Abs{g}_\infty,\Abs{g'}_\infty\}$ is a Banach space,
where $\Abs{\cdot}_\infty$ denotes the uniform norm.
Also note that $\mathring M$ is convex and therefore $g\in C_{b}^1(\mathring M)$
is Lipschitz continuous.
Hence, $g$ has a unique continuous extension on $\overline M$ and therefore
$g(0,y)$ with $y\in\set R$ is well-defined (with the convention of using the
same symbol for the extension).
That is why, we will just write $g\in\Cone$ from now on.

\begin{definition}
	For $g\in\Cone$, we define the \emph{coupled skinny baker's map} as
	\[
		\Bc:M^2\subset\set R^4\to M^2:(x,y,z,w)\mapsto
			(B_\alpha(x,y),B_\beta(z,w)+(0,g(x,y))).
	\]
\end{definition}

The uncoupled and coupled skinny baker's map are piecewise $C^1$ dynamical systems,
where the domain $M^2$ has the partition 
\begin{align*}
	U_1&:=\left[0,\frac{1}{2}\right)\times\set R
		\times\left[0,\frac{1}{2}\right)\times\set R,\quad
	U_2:=\left[0,\frac{1}{2}\right)\times\set R
		\times\left[\frac{1}{2},1\right)\times\set R,\\
	U_3&:=\left[\frac{1}{2},1\right)\times\set R
		\times\left[\frac{1}{2},1\right)\times\set R,\quad
	U_4:=\left[\frac{1}{2},1\right)\times\set R
		\times\left[0,\frac{1}{2}\right)\times\set R
\end{align*}
and
\[
	M^2_0=\{(x,y,z,w)\in M^2:x\textnormal{ or }z
		\textnormal{ is a dyadic fraction or }0\},
\]
with $\lambda^4(M^2_0)=0$.
Furthermore, we have for the derivative (piecewise, i.e.\ on each $\mathring U_i$)
\begin{align}
	\label{derivative}
	\Bc'(x,y,z,w)=
	\begin{pmatrix}
		2		&	0		&	0	&	0\\
		0		&	\alpha	&	0	&	0\\
		0		&	0		&	2	&	0\\
		g_x(x,y)&	g_y(x,y)& 	0	&	\beta
	\end{pmatrix}
\end{align}
and the derivative is (piecewise) bounded because $g\in\Cone$ (note that by
$g_x$ we mean the partial derivative of $g$ with respect to the variable $x$).

A direct calculation gives that the Lyapunov exponents of the uncoupled skinny 
baker's map are $\log 2$, $\log 2$, $\log\alpha$ and $\log\beta$ almost
everywhere with respect to the Lebesgue measure and the physical measure.
This implies for $\alpha\leq\beta$
\[
	D_L(\muuc)=
	\begin{cases}
		2-2\frac{\log 2}{\log\beta}&\textnormal{ if }\beta\leq\frac{1}{4}\\
		3-2\frac{\log 2}{\log\alpha}-\frac{\log\beta}{\log\alpha}&
			\textnormal{ if }\beta\geq\frac{1}{4}
	\end{cases}
\]
(for $\alpha\geq\beta$ interchange $\alpha$ with $\beta$).
Observe that $d(\muuc)=D_1(\muuc)<D_L(\muuc)$ for $\alpha\neq\beta$ and
$d(\muuc)=D_1(\muuc)=D_L(\muuc)$ for $\alpha=\beta$, i.e.\ the Kaplan-Yorke
equality fails for Lebesgue almost every
$(\alpha,\beta)\in\big(0,\frac{1}{2}\big)^2$.

In Figure \ref{relation_of_information_dimension_Lyapunov_dimension}, we show
the information dimension and Lyapunov dimension of the uncoupled skinny 
baker's map for fixed $\alpha=0.05$.
Note the relation between the two branches of the Lyapunov dimension
for $\beta\geq 0.05$.

\begin{figure}[H]
	\includegraphics[width=1.0\textwidth]{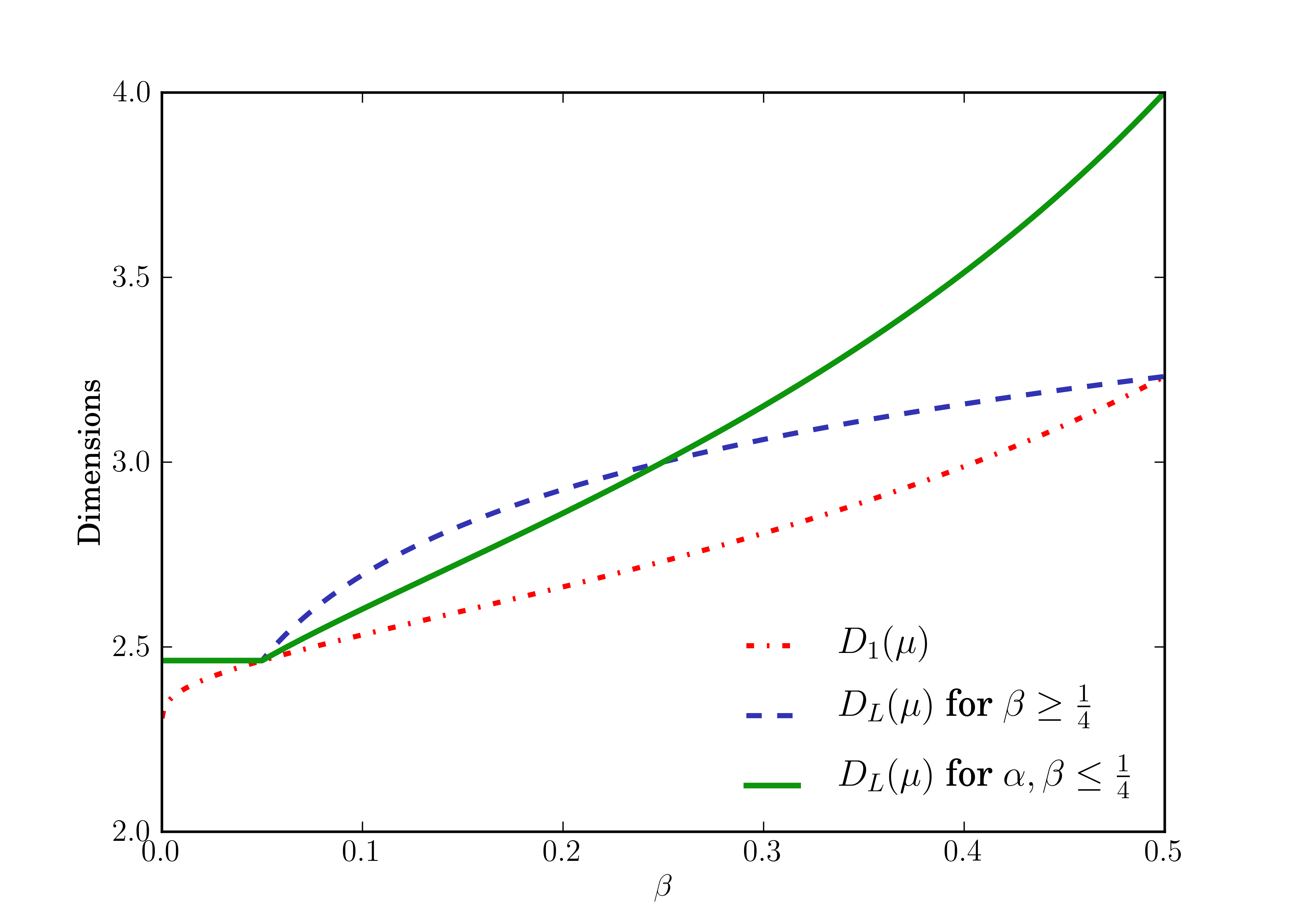}
	\caption{The information dimension $D_1(\muuc)$ and the Lyapunov dimension
		$D_L(\mu)$ for the uncoupled skinny baker's map for fixed $\alpha=0.05$.}
	\label{relation_of_information_dimension_Lyapunov_dimension}
\end{figure}

\begin{proposition}
	\label{proposition_conjugation}
	For $g\in\Cone$, the attractor $\Ac$ of the coupled skinny baker's map is 
	the image of a map $\hc$, where the domain of $\hc$ is the attractor of
	the uncoupled skinny baker's map. 
	The conjugacy $\hc:\Auc\to \Ac$ is defined as
	\begin{align*}
		\hc[\pm 1](x,y,z,w):=\left(x,y,z,w\pm\sum_{i=0}^{\infty}\beta^i 
			g\left(B_\alpha^{-i-1}(x,y)\right)\right).
	\end{align*}
	We have $\Buc(x,y,z,w) = \hc[-1](\Bc(\hc(x,y,z,w)))$ and furthermore, for
	\begin{enumerate}
		\item[(i)] $\alpha>\beta$: $\hc$ is bi-Lipschitz,
		\item[(ii)] $\alpha=\beta$: $\hc[\pm 1]$ is Hölder continuous for
			all Hölder exponents $\rho<1$,
		\item[(iii)] $\alpha<\beta$: $\hc$ is Hölder continuous with 
			Hölder exponent $\rho=\frac{\log\beta}{\log\alpha}$.
	\end{enumerate}
\end{proposition}
\begin{proof}
	For $a\in[0,1]$ let $((a)_i)_{i\in\set N_0}$ be the binary representation
	of $a$ with $(a)_i$ zero or one for all $i\in\set N_0$ and
	\[
		a=\sum\limits_{i=0}^{\infty}\frac{(a)_i}{2^{i+1}}.
	\]
	In the case that $a$ is a dyadic fraction, we have to make a choice 
	between the terminating or non-terminating representation to ensure the
	uniqueness of the binary representation.
	For $n\in\set N_0$ and $(x,y,w,z)\in M^2$, define
	\begin{align}\label{definition_xyzw_n}
		(x_n,y_n,z_n,w_n):=\Bc^n(x,y,z,w)
	\end{align}
	and a direct calculation shows
	\begin{align*}
		x_n&=2^n x-\sum\limits_{i=0}^{n-1}2^{n-1-i}(x)_i,\quad
		y_n=\alpha^n y+(1-\alpha)\sum\limits_{i=0}^{n-1}\alpha^{n-1-i}(x)_i,\\
		z_n&=2^n z-\sum\limits_{i=0}^{n-1}2^{n-1-i}(z)_i,\quad
		w_n=\beta^n w+(1-\beta)\sum\limits_{i=0}^{n-1}\beta^{n-1-i}(z)_i
			+\sum\limits_{i=0}^{n-1}\beta^{n-1-i}g(x_i,y_i).
	\end{align*}

	First, we want to explain that the series in the definition of $\hc$ is
	well-defined: for $B_{\alpha}$ restricted to $[0,1)\times[0,1]$ we can
	define an inverse map by
	\begin{align*}
		B_\alpha^{-1}(x,y):=
		\begin{cases}
			\left(\frac{x}{2},\frac{y}{\alpha}\right)
				&\textnormal{ if }y\leq\frac{1}{2}\\\\
			\left(\frac{x+1}{2},\frac{y-(1-\alpha)}{\alpha}\right)
				&\textnormal{ if }y>\frac{1}{2}
		\end{cases},
	\end{align*}
	on $[0,1)\times([0,\alpha]\cup[1-\alpha,1])$.
	Since $[0,1)\times A_{\alpha}$ is the attractor of $B_\alpha$, we have
	\begin{align*}
		(x,y)&\in[0,1)\times A_{\alpha}=B_{\alpha}([0,1)\times A_{\alpha})\\
		&\Longleftrightarrow(x,y)\in [0,1)\times A_{\alpha}, 
			B_{\alpha}([0,1)\times A_{\alpha}), 
			B_{\alpha}^{2}([0,1)\times A_{\alpha}),
			\ldots
	\end{align*}
	and this means
	\[
		\ldots,B_{\alpha}^{-2}(x,y),B_{\alpha}^{-1}(x,y),(x,y)
			\in[0,1)\times A_{\alpha},
	\]
	i.e.\ each $(x,y)\in[0,1)\times A_{\alpha}$ has a unique infinite past 
	history
	\begin{align}\label{definition_xy_n_past_history}
		(x_{-n}, y_{-n}):=B_\alpha^{-n}(x,y),
	\end{align}
	with $n\in\set N$ on the attractor.
	Thus, for $(x,y)\in[0,1)\times A_{\alpha}$ the occurring series in $\hc$ is 
	well-posed and convergent, since $g$ is bounded.
	Furthermore, for each $n\in\set N$
	\[
		y=\alpha^n y_{-n}+(1-\alpha)\sum_{j=0}^{n-1}\alpha^j(x_{-j-1})_0,
	\]
	where $j=n-1-i$ and $(x_{-n})_{n-1-j}=(x_{-j-1})_0$. 
	Therefore, in the limit,
	\[
		y=(1-\alpha)\sum_{i=0}^{\infty}\alpha^i (x_{-i-1})_{0}
	\]
	and we have the analog result for $(z,w)\in[0,1)\times A_{\beta}$ (for 
	the uncoupled system).
	We will need this result and a similar argument to show that $\hc(\Auc)$ is
	the attractor for the coupled skinny baker's map: direct computation gives
	for $(x,y,z,w)\in[0,1)\times A_{\alpha}\times M$ 
	\begin{equation}\label{conjugation}
		 \hc((\Buc(x,y,z,w))=\Bc(\hc(x,y,z,w))
	\end{equation}
	and we will see that only $(x,y,z,w)\in A$ matter.
	For some $\delta>0$, define the set
	\[
		V:=[0,1)\times[-\delta,1+\delta]\times[0,1)\times
			\left[-\delta-\frac{\supnorm g}{1-\beta},
			1+\delta+\frac{\supnorm g}{1-\beta}\right].
	\]
	Note that $\Bc(V)\subset V$, i.e.\ we can define
	\[
		\Ac:=\bigcap\limits_{n\in\mathbb N}\Bc^n(V).
	\]
	Also note that $\hc(\Auc)\subset V$ and due to \eqref{conjugation}, we have 
	$\hc(\Auc)=\Bc^n(\hc(\Auc))\subset \Bc^n(V)$ for each $n\in\set N$, i.e.\
	$\hc(\Auc)\subset\Ac$.
	Now, using a similar argument as above, we will show $\hc(\Auc)=\Ac$. 
	We have that each $(x,y,z,\tilde w)\in \Ac$ has at least one past history, 
	denoted by $(x_{-1},y_{-1},z_{-1},\tilde w_{-1})$ and this point has also at
	least one past history, denoted by $(x_{-2},y_{-2},z_{-2},\tilde w_{-2})$ 
	and so on.
	Every single possible, infinite past history is a subset of $V$, i.e.\ is 
	bounded.
	For each $n\in\set N$ we have
	\[
		\tilde w=\beta^n\tilde w_{-n}
			+(1-\beta)\sum_{j=0}^{n-1}\beta^j(z_{-j-1})_0
			+\sum_{j=0}^{n-1}\beta^j g\left(x_{-j-1},y_{-i-1}\right),
	\]
	again $j=n-1-i$ and $(z_{-n})_{n-1-j}=(z_{-j-1})_0$.
	Thus, in the limit,
	\[
		\tilde w=(1-\beta)\sum_{i=0}^{\infty}\beta^i (z_{-i-1})_{0}
			+\sum_{i=0}^{\infty}\beta^i g\left(B_\alpha^{-i-1}(x,y)\right)
		=w+\sum_{i=0}^{\infty}\beta^i g\left(B_\alpha^{-i-1}(x,y)\right),
	\]
	with $(x,y,z,w)\in A$.
	This shows $\hc(\Auc)=\Ac$. 
	It also follows that $\hc$ is injective, i.e.\ the inverse map 
	$\hc[-1]:\Ac\to\Auc$ exists and is given as stated above.

	Now, we prove the continuity properties of $\hc$.
	Let $(x,y,z,w),(\tilde x,\tilde y,\tilde z,\tilde w)\in\Auc$ and 
	\begin{align*}
		\Abs{\hc(x,y,z,w)-\hc(\tilde x,\tilde y,\tilde z,\tilde w)}
		&=\Abs{(x-\tilde x,y-\tilde y,z-\tilde z,w-\tilde w+\Delta_g)}\\
		&\leq\abs{x-\tilde x}+\abs{y-\tilde y}+\abs{z-\tilde z}
			+\abs{w-\tilde w}+\abs{\Delta_g},
	\end{align*}
	with
	\[
		\Delta_g=\sum_{i=0}^{\infty}\beta^i
			\left[g\left(B_\alpha^{-i-1}(x,y)\right)
			-g\left(B_\alpha^{-i-1}(\tilde x,\tilde y)\right)\right].
	\]
	The estimation of $\Delta_g$ will be the main task.
	First, note that
	\begin{equation}\label{Delta_g_estimation}
		\abs{\Delta_g}\leq\Lip g
			\sum_{i=0}^{\infty}\beta^i\Abs{B_\alpha^{-i-1}(x,y)
			-B_\alpha^{-i-1}(\tilde x,\tilde y)}.
	\end{equation}
	For $y=\tilde y$,
	\[
		\Abs{B_\alpha^{-i-1}(x,y)-B_\alpha^{-i-1}(\tilde x,\tilde y)}
		\leq \abs{x-\tilde x}
	\]
	because
	\[
		B_\alpha^{-i-1}(x,y)=\left(x_{-i-1},y_{-i-1}\right)
		=\left(\frac{x_{-i}+(y_{-i})_0}{2},
			\frac{y_{-i}-(1-\alpha)(y_{-i})_0}{\alpha}\right),
	\]
	and therefore
	\[
		y_{-i-1}-\tilde y_{-i-1}=0
		\quad\textrm{ and }\quad
		x_{-i-1}-\tilde x_{-i-1}=\frac{x-\tilde x}{2^{i+1}}.
	\]
	For $\abs{y-\tilde y}\geq (1-2\alpha)$, we have
	\[
		\Abs{B_\alpha^{-i-1}(x,y)-B_\alpha^{-i-1}(\tilde x,\tilde y)}<2
		\leq \frac{2}{1-2\alpha}\abs{y-\tilde y}.
	\]
	That means, for this two cases,
	\[
		\abs{\Delta_g}\leq\frac{\Lip g}{1-\beta}\left[\abs{x-\tilde x}
			+\frac{2}{1-2\alpha}\abs{y-\tilde y}\right].
	\]
	To handle the remaining case, consider for each $k\in\set N_0$
	\begin{equation}\label{y_slices}
		(1-2\alpha)\alpha^k>\abs{y-\tilde y}\geq (1-2\alpha)\alpha^{k+1}.
	\end{equation}
	If $y$ and $\tilde y$ satisfy this condition for a $k\in\set N_0$, then
	they stay close together (in the sense that they are both in $[0,\alpha]$
	or in $[1-\alpha,1]$) for $k$ steps backward and we can split the sum
	in \eqref{Delta_g_estimation} into the first $k$ parts and the remaining 
	part.
	We get the following upper bound
	\begin{equation}\label{k_estimations}
		\frac{\abs{\Delta_g}}{\Lip g}
		\leq\sum\limits_{i=0}^k\frac{\beta^i}{2^{i+1}}\abs{x-\tilde x}
			+ \sum\limits_{i=0}^k\frac{\beta^i}{\alpha^{i+1}}\abs{y-\tilde y}
			+\frac{2}{1-\beta}\beta^{k+1}.
	\end{equation}
	The factor in front of $\abs{x-\tilde x}$ is smaller than 
	$\frac{1}{2-\beta}<\frac{1}{1-\beta}$, independent of $k$ and the 
	relation between $\alpha$ and $\beta$.

	For $\alpha>\beta$ we can get $k$-independent estimates for the other 
	two terms in \eqref{k_estimations},
	\[
		\sum\limits_{i=0}^k\frac{\beta^i}{\alpha^{i+1}}\abs{y-\tilde y}
		\leq\frac{1}{\alpha-\beta}\abs{y-\tilde y}
		\quad\textrm{ and }\quad
		\frac{2}{1-\beta}\beta^{k+1}
		\leq\frac{2}{(1-\beta)(1-2\alpha)}\abs{y-\tilde y},
	\]
	where we used \eqref{y_slices} for the third term.
	Hence, for all $(x,y,z,w),(\tilde x,\tilde y,\tilde z,\tilde w)\in \Auc$
	\begin{align*}
		\Abs{\hc(x,y,z,w)-\hc(\tilde x,\tilde y,\tilde z,\tilde w)}
		\leq&\left[1+\frac{\Lip g}{1-\beta}\right]\abs{x-\tilde x}\\
		+&\left[1+\Lip g\left[\frac{1}{\alpha-\beta}+
			\frac{2}{(1-\beta)(1-2\alpha)}\right]\right]\abs{y-\tilde y}\\
		+&\abs{z-\tilde z} + \abs{w-\tilde w}.
	\end{align*}
	The same result applies to $\hc[-1]$, i.e.\ $\hc$ is bi-Lipschitz for
	$\alpha>\beta$.

	For $\alpha=\beta$, using \eqref{y_slices}, the second term in 
	\eqref{k_estimations} is bounded above by
	\[
		\sum\limits_{i=0}^k\frac{\beta^i}{\alpha^{i+1}}\abs{y-\tilde y}
		=\frac{k+1}{\alpha}\abs{y-\tilde y}
		<\frac{1}{\alpha}\abs{y-\tilde y}
		+\frac{1}{\alpha\log\alpha}\abs{y-\tilde y}\log\abs{y-\tilde y}.
	\]
	For $\rho\in[0,1)$ we can always find a constant $C(\rho)>0$ such that
	\[
		-\abs{y-\tilde y}^{1-\rho}\log\abs{y-\tilde y}\leq C(\rho),
	\]
	i.e.\ for all $(x,y,z,w),(\tilde x,\tilde y,\tilde z,\tilde w)\in \Auc$
	\begin{align*}
		\Abs{\hc(x,y,z,w)-\hc(\tilde x,\tilde y,\tilde z,\tilde w)}
		\leq&\left[1+\frac{\Lip g}{1-\alpha}\right]\abs{x-\tilde x}\\
		+&\left[1+\frac{\Lip g}{\alpha}\left[1-\frac{C(\rho)}{\log\alpha}+
			\frac{2}{(1-\alpha)(1-2\alpha)}\right]\right]\abs{y-\tilde y}^\rho\\
		+&\abs{z-\tilde z} + \abs{w-\tilde w}.
	\end{align*}
	Again, the same result applies to $\hc[-1]$, i.e.\ $\hc$ and $\hc[-1]$ are 
	Hölder continuous for all Hölder exponents $\rho<1$ for $\alpha=\beta$.
	
	The last case is $\alpha<\beta$.
	Set
	\[
		\rho:=\frac{\log\beta}{\log\alpha}.
	\]
	We observe first that \eqref{y_slices} is equivalent to
	\begin{equation}\label{y_slices_equivalent}
		(1-2\alpha)^\rho\beta^k>\abs{y-\tilde y}^\rho
		\geq(1-2\alpha)^\rho\beta^{k+1}
	\end{equation}
	and therefore
	\[
		\left(\frac{\beta}{\alpha}\right)^{k+1}\abs{y-\tilde y}
		\leq\frac{(1-2\alpha)^{1-\rho}}{\alpha}\abs{y-\tilde y}^{\rho}.
	\]
	Furthermore, by using the relation
	\begin{equation*}
		\sum\limits_{i=0}^k q^k=\frac{q^{k+1}-1}{q-1},
	\end{equation*}
	for $q\neq 1$, we get
	\[
		\sum\limits_{i=0}^k\frac{\beta^i}{\alpha^{i+1}}\abs{y-\tilde y}
		<\frac{1}{\beta-\alpha}
			\left(\frac{\beta}{\alpha}\right)^{k+1}\abs{y-\tilde y}
		\leq\frac{(1-2\alpha)^{1-\rho}}{\alpha(\beta-\alpha)}
			\abs{y-\tilde y}^\rho.
	\]
	Hence, using \eqref{y_slices_equivalent} for the last term in 
	\eqref{k_estimations}, for all 
	$(x,y,z,w),(\tilde x,\tilde y,\tilde z,\tilde w)\in \Auc$
	\begin{align*}
		\Abs{\hc(x,y,z,w)-\hc(\tilde x,\tilde y,\tilde z,\tilde w)}
		\leq&\left[1+\frac{\Lip g}{1-\beta}\right]\abs{x-\tilde x}\\
		+&\left[1+\frac{\Lip g}{(1-2\alpha)^\rho}
			\left[\frac{1}{\alpha(\beta-\alpha)}+
			\frac{2}{(1-\beta)(1-2\alpha)}\right]\right]\abs{y-\tilde y}^\rho\\
		+&\abs{z-\tilde z} + \abs{w-\tilde w}
	\end{align*}
	(note that $(1-2\alpha)^{1-\rho}\leq(1-2\alpha)^{-\rho}$).
\end{proof}

There are two direct consequences of this proposition.
The first one pinpoints the physical measure $\muc$ of the coupled skinny
baker's map.
With the conjugacy $\hc$ we have a natural candidate, namely, the image
measure of $\muuc=\mu_\alpha\times\mu_\beta$.
The second consequence is that for $\alpha\geq\beta$ all the dimensions of the
attractor and the physical measure of the coupled and uncoupled system coincide.

\begin{lemma}
	For $g\in\Cone$ the image measure $\hc(\muuc)=:\muc$ is invariant under 
	$\Bc$ and is ergodic.
	It is the unique physical measure of $\Bc$, too.
\end{lemma}
\begin{proof}
	Since $\hc$ is a conjugacy on the support of $\muuc$, it follows that 
	$\muc$ is invariant under $\Bc$ and ergodic.
	Now, we want to prove that $\muc$ is the unique physical measure.
	We will show that \eqref{property_physical_measure} holds for Lebesgue
	almost every point.
	It is enough to prove \eqref{property_physical_measure} for all continuous
	functions with compact support \cite[Theorem 29.12 and Corollary 30.9]{Bauer1992}.
	Using \cite[Lemma 6.13]{Walters1982} and \cite[Lemma 31.4]{Bauer1992},
	we can find a set $W\in\balg{M^2}$ with $\muc(W)=1$ such that for all
	$(x,y,z,w)\in W$ and for every continuous function $\varphi:M^2\to\set R$
	with compact support,
	\begin{align}
		\label{physical_measure_condition}
		\lim\limits_{N\to\infty}\frac{1}{N}\sum\limits_{n=0}^{N-1}
			\varphi(\Bc^n(x,y,z,w))
		=\int\varphi d\muc.
	\end{align}
	Furthermore, for $(x,y,z,w),(x,\tilde y,z,\tilde w)\in M^2$ we have
	\[
		\lim\limits_{n\to\infty}\Abs{\Bc^n(x,y,z,w)-\Bc^n(x,\tilde y,z,\tilde w)}=0
	\]
	because (using the notation \eqref{definition_xyzw_n} from the previous proof)
	\[
		\abs{y_n-\tilde y_n}=\alpha^n\abs{y-\tilde y}
	\]
	and
	\[
		\abs{w_n-\tilde w_n}\leq\beta^n\abs{w-\tilde w}+\Lip{g}\abs{y-\tilde y}\cdot
		\begin{cases}
			\frac{\alpha^n}{\alpha-\beta}&\textnormal{ if }\alpha>\beta\\
			\alpha^{n-1}n&\textnormal{ if }\alpha=\beta\\
			\frac{\beta^n}{\beta-\alpha}&\textnormal{ if }\alpha<\beta
		\end{cases}.
	\]
	Therefore, we have for every continuous function $\varphi$ with compact
	support
	\[
		\lim\limits_{n\to\infty}\abs{\varphi(\Bc^n(x,y,z,w))
		-\varphi(\Bc^n(x,\tilde y,z,\tilde w))}=0,
	\]
	since $\varphi$ is also uniformly continuous.
	Hence, by using the Cesàro mean, we get
	\[
		\lim\limits_{N\to\infty}\abs{\frac{1}{N}\sum\limits_{n=0}^{N-1}
		\left[\varphi(\Bc^n(x,y,z,w))-
		\varphi(\Bc^n(x,\tilde y,z,\tilde w))\right]}=0.
	\]
	This means condition \eqref{physical_measure_condition} is fulfilled
	for all $(x,y,z,w)\in M^2$ with $(x,z)\in\pi_{x,z}(W)$, where $\pi_{x,z}$ is
	the canonical projection onto $[0,1)^2\subset\set R^2$.
	Note that $\pi_{x,z}(W)$ is Lebesgue measurable 
	\cite[Theorem 1.7.19 and Theorem 1.7.9]{KrantzParks2008}.
	Furthermore,
	\begin{align*}
		1&=\muc(W)\leq
		\mu(\{(x,y,z,w)\in ([0,1)\times[0,1])^2:(x,z)\in\pi_{x,z}(W)\})\\
		&=\lambda^2(\pi_{x,z}(W))\nu_\alpha([0,1])\nu_\beta([0,1])
		=\lambda^2(\pi_{x,z}(W))\leq 1,
	\end{align*}
	i.e.\ $\pi_{x,z}(W)$ has full measure in $[0,1)^2$.
	This and the property that $\lambda^4$ is $\sigma$-finite implies that the
	set of exceptions of \eqref{physical_measure_condition}	has zero measure.
	Therefore, for every continuous function $\varphi:M^2\to\set R$ with compact
	support,
	\[
		\lim\limits_{N\to\infty}\frac{1}{N}\sum\limits_{n=0}^{N-1}
			\varphi(\Bc^n(x,y,z,w))
		=\int\varphi d\muc,
	\]
	for almost every $(x,y,z,w)\in M^2$ with respect to the Lebesgue measure.
\end{proof}

Since the uncoupled case corresponds to $g=0$, we also get that 
$\muuc=\mu_\alpha\times\mu_\beta$ is the unique physical measure for the
uncoupled skinny baker's map $\Buc$.

The second consequence of Proposition \ref{proposition_conjugation} follows
directly for $\alpha>\beta$ from the bi-Lipschitz continuity of the conjugacy
and for $\alpha=\beta$ from the Hölder continuity of the conjugacy and its
inverse with arbitrary Hölder exponent $\rho<1$.

\begin{corollary}
	For $\alpha\geq\beta$ and $g\in\Cone$ we have that	$\muc$ is exact 
	dimensional with $d(\muc)=D_1(\muc)=d(\muuc)$.
	We get $D_B(\Ac)=D_B(\Auc)$, too (note that also other dimensions of the
	the physical measure and the attractor coincide).
\end{corollary}

Now, we calculate the Lyapunov exponents of the coupled skinny baker's map.

\begin{lemma}
	\label{lemma_lyapunov_exponents}
	For $g\in\Cone$ the Lyapunov exponents of the coupled skinny baker's map
	$\Bc$ are $\log 2$, $\log 2$, $\log\alpha$ and $\log\beta$ almost everywhere
	with respect to the Lebesgue measure and the physical measure.
\end{lemma}
\begin{proof}
	We consider points $(x,y,z,w)\in M^2\backslash M^2_{0}$.
	We have $\lambda^4(M^2_0)=0$ and $\muc(M^2_0)=0$, too.
	To compute directly the Lyapunov exponents at a point we have to find lower
	and upper bounds for $\Abs{(\Bc^n)'(x,y,z,w) v}$ with $v\in\set R^4$.
	The derivative $(\Bc^n)'$ for all $n\in\set N$ is given by
	\[
		(\Bc^n)'(x,y,z,w)=\prod\limits_{k=1}^n\Bc'({\Bc^{n-k}(x,y,z,w)}).
	\]
	Using \eqref{derivative}, we get
	\[
		(\Bc^n)'(x,y,z,w)=
		\begin{pmatrix}
			2^n			&	0				&	0	&	0\\
			0			&	\alpha^n		&	0	&	0\\
			0			&	0				&	2^n	&	0\\
			a_n(x,y)	&	b_n(x,y)		& 	0 	&	\beta^n
		\end{pmatrix},
	\]
	with
	\[
		a_n(x,y)=\sum\limits_{i=0}^{n-1}2^i\beta^{n-1-i}g_x(\Buc_\alpha^i(x,y))
		\quad\textnormal{ and }\quad
		b_n(x,y)=\sum\limits_{i=0}^{n-1}\alpha^i\beta^{n-1-i}g_y(\Buc_\alpha^i(x,y)).
	\]
	Observe that $\log 2$ is a Lyapunov exponent with at least multiplicity 2.
	This is true because
	\[
		\Abs{(\Bc^n)'(x,y,z,w)v}^2=2^{2n}+(a_n(x,y)\cdot c)^2\leq
		(1+\Abs{g_x}_{\infty}^2 c^2)2^{2n}
	\]
	for all $v=(c,0,d,0)$, $\Abs{v}=1$.
	We get the Lyapunov exponent $\log\beta$ for $v=(0,0,0,1)$.
	For $\alpha\geq\beta$ and $v=(0,1,0,0)$ we get the last Lyapunov exponent
	$\log\alpha$ because 
	\[
		\abs{b_n(x,y)}\leq\Abs{g_y}_{\infty}\alpha^{n-1}
		\begin{cases}
			\frac{\alpha}{\alpha-\beta}&\textnormal{ if }\alpha>\beta\\
			n&\textnormal{ if }\alpha=\beta
		\end{cases}
	\]
	(for $\alpha=\beta$ we have for all $v=(0,c,0,d)$, $\Abs{v}=1$, $c\neq 0$
	that $\Abs{(\Bc^n)'(x,y,z,w)v}>\alpha^n\abs{c}$ and therefore $\log\alpha$
	has multiplicity 2).
	For $\alpha<\beta$ set $b_{\infty}(x,y):=\lim\limits_{n\to\infty}b_n(x,y)/\beta^n$
	and	use $v=(0,1,0,-b_{\infty}(x,y))$.
	Again, we get $\log\alpha$ because
	\[
		\abs{b_n(x,y)-\beta^n b_{\infty}(x,y)}
		\leq\frac{\Abs{g_y}_{\infty}}{\beta-\alpha}\alpha^n.
	\]
	All this together proves the assertion.
\end{proof}

Note that the $w$-direction corresponds to the smallest Lyapunov exponent for
$\alpha>\beta$, but not for $\alpha<\beta$.

\begin{corollary}\label{corollary_D_L}
	The Lyapunov dimension of the coupled system coincides with the Lyapunov
	dimension of the uncoupled system for every $g\in\Cone$, i.e.\
	\[
		D_L(\muc)=D_L(\muuc)=:D_L.
	\]
	This implies for $\alpha\geq\beta$ that the relation between
	the	information	dimension and Lyapunov dimension of the coupled	system is the
	same as for the uncoupled system, i.e $D_1(\muc)<D_L$ for $\alpha>\beta$ and 
	$D_1(\muc)=D_L$ for $\alpha=\beta$.
\end{corollary}

Now, what is left is the case $\alpha<\beta$.
With the next 4 assertions we will prove that $d(\muc)=D_1(\muc)=D_L$ in a
prevalent sense, where we will mainly use the potential theoretic method for the
pointwise dimension, see Theorem \ref{pointwise_dimension_equal_sup_s_potential}.
The main step will be to establish a lower bound for the pointwise dimension,
and in order to this we will show that for a prevalent set of $g$'s the
following is true: for all $0\leq s<D_L$, the $s$-potential $\varphi_s(\muc,v)$
is finite for \mucae{$v$} in the attractor $\Ac$ (defined in Proposition 
\ref{proposition_conjugation}).
The corresponding set of exceptions is
\begin{multline}
	\label{shy_set}
	E:=\{g\in\Cone:\textnormal{ there are }0\leq s<D_L \textnormal{ and }\\
		W\subset\Ac\textnormal{ with }\muc(W)>0
		\textnormal{ such that }\varphi_s(\muc,v)=\infty
		\textnormal{ for all }v\in W\}.
\end{multline}
In Theorem \ref{prevelant_result} we will embed $E$ into a bigger Borel set and
we will show that Lebesgue measure supported on the finite dimensional subspace
\begin{align}
	\label{definition_probe}
	P:=\{(x,y)\mapsto \lambda\cdot p(x,y):\lambda\in\set R\}\subset\Cone
\end{align}
is transverse to this bigger set, where we require $p(x,y)=y$ for all 
$(x,y)\in[0,1)\times[0,1]$.

Before we proceed, we want to motivate why we are choosing this prevalent
setting.
To use Theorem \ref{pointwise_dimension_equal_sup_s_potential}, we have to 
estimate
\begin{align*}
	\varphi_s(\muc,v)
	=\int\limits_{\Ac}\frac{1}{\abs{v-\tilde v}^s}d\muc(\tilde v)
	&=\int\limits_{\Auc}\frac{1}{\abs{\hc(x,y,z,w)-
		\hc(\tilde x,\tilde y,\tilde z,\tilde w)}^s}
		d\muuc(\tilde x,\tilde y,\tilde z,\tilde w)\\
	&=\int\limits_{\Auc}\frac{1}{(\Iuc^2+\Ic^2)^{s/2}}
		d\muuc(\tilde x,\tilde y,\tilde z,\tilde w),
\end{align*}
for $v\in\Ac$, $(x,y,z,w)=\hc[-1](v)$ and with
\begin{align*}
	\Iuc^2 &:=(x-\tilde x)^2+(y-\tilde y)^2+(z-\tilde z)^2,\\
	\Ic &:= w-\tilde w+ 
	\sum_{i=0}^{\infty}\beta^i\left[ g\left(B_\alpha^{-i-1}(x,y)\right)
	-g\left(B_\alpha^{-i-1}(\tilde x,\tilde y)\right)\right].
\end{align*}
It is very difficult to estimate this integral for an arbitrary $g\in\Cone$.
The notion of prevalence comes here into play, namely, by adding a linear
perturbation term to $g$,
\begin{align*}
	\gp(x,y):=g(x,y)+\lambda\cdot p(x,y),
\end{align*}
for $\lambda\in\set R$ and estimating the following integral
\begin{align}
	\label{main_integral}
	\int\limits_{-\lambda_0}^{\lambda_0}
		\varphi_s(\mup,\hcp(v))d\lambda
	&=\int\limits_{-\lambda_0}^{\lambda_0}\;\int\limits_{\Auc}\frac{1}{(\Iuc^2+
		(\Ic+\lambda\Ip)^2)^{s/2}}d\muuc(\tilde x,\tilde y,\tilde z,\tilde w)
		d\lambda,
\end{align}
for $v=(x,y,z,w)\in\Auc$ and (using the notation
\eqref{definition_xy_n_past_history} from the proof of Proposition
\ref{proposition_conjugation})
\[
	\Ip:=\sum_{i=0}^{\infty}\beta^i\left[y_{-i-1}-\tilde y_{-i-1}\right].
\]
Now, the main ingredient is to use Tonelli's theorem. 
As a result, the order of integration in \eqref{main_integral} can be interchanged
and the estimation of the inner integral over $\lambda$ is feasible.

The next proposition will be the major technical step in the proof of Theorem
\ref{prevelant_result}. 
But first, we need the following lemma.

\begin{lemma}
	\label{lemma_I_p_estimations}
	Assume $\alpha<\beta$.
	For sufficient small $\abs{y-\tilde y}$ with
	$y,\tilde y\in A_\alpha$ we have that $\Ip$ is bounded from above and below,
	\[
		K_1\abs{y-\tilde y}^\rho\geq\abs{\Ip}\geq K_2\abs{y-\tilde y}^\rho,
	\]
	with $\rho=\frac{\log\beta}{\log\alpha}$ and $K_1,K_2>0$.
\end{lemma}
\begin{proof}
	We can use the estimate of $\Delta_g$ from the proof of Proposition
	\ref{proposition_conjugation}, with $g$ replaced by $y$, for the upper bound.
	Choosing $k$ so that \eqref{y_slices} holds, we have for the lower bound
	\begin{align}\label{I_p_lower_bound}
		\abs{\Ip}&\geq\abs{
			\abs{\sum_{i=0}^{k}\beta^i\left[y_{-i-1}-
			\tilde y_{-i-1}\right]}
			-\abs{\sum_{i=k+1}^{\infty}\beta^i\left[y_{-i-1}-
			\tilde y_{-i-1}\right]}}\notag\\
		&=\abs{\frac{\abs{y-\tilde y}}{\beta-\alpha}\left(
			\left(\frac{\beta}{\alpha}\right)^{k+1}-1\right)-\abs{R}},
	\end{align}
	with
	\[
		R:=\sum_{i=k+1}^{\infty}\beta^i\left[y_{-i-1}-\tilde y_{-i-1}\right].
	\]
	Also, from \eqref{y_slices} and \eqref{y_slices_equivalent} follows
	\[
		\left(\frac{\beta}{\alpha}\right)^{k+1}\abs{y-\tilde y}
		\geq\beta(1-2\alpha)^{1-\rho}\abs{y-\tilde y}^{\rho}.
	\]
	Furthermore, using \eqref{y_slices} again,
	\[
		\frac{\abs{y-\tilde y}}{\beta-\alpha}\left(
			\left(\frac{\beta}{\alpha}\right)^{k+1}-1\right)-\abs{R}
		\geq\frac{\abs{y-\tilde y}}{\beta-\alpha}\left(
			(1-\alpha)(1-2\beta)\left(\frac{\beta}{\alpha}\right)^{k+1}-1\right).
	\]
	By choosing $k$ large enough,
	\[
		\frac{(1-\alpha)(1-2\beta)}{2}\left(\frac{\beta}{\alpha}\right)^{k+1}>1,
	\]
	i.e.\ $\abs{y-\tilde y}$ sufficient small, the term in the absolute value in 
	\eqref{I_p_lower_bound} is always positive and we have
	$\abs{\Ip}\geq K_2\abs{y-\tilde y}^{\rho}$ where
	\[
		K_2:=\frac{\beta(1-2\alpha)^{1-\rho}(1-\alpha)(1-2\beta)}{2(\beta-\alpha)}.
		\qedhere
	\]
\end{proof}

\begin{proposition}
	\label{main_proposition}
	Suppose $\alpha<\beta$ and $g\in\Cone$.
	For $1<s<D_L$, $\lambda_0>0$ and $v\in\Auc$ we have
	\begin{align*}
		\int\limits_{-\lambda_0}^{\lambda_0}\varphi_s(\mup,\hcp(v))d\lambda
		\leq& C_{I_0}+\sum\limits_{k=k_0}^{\infty}
			\Bigg( C_{L_1}\alpha^{k(2-2\frac{\log 2}{\log\beta}-s)}+\\
		&k\,C_{L_2}\max\left\{\alpha^{k(2-2\frac{\log 2}{\log\beta}-s)},
			\alpha^{k\frac{s}{s-1+\rho}(3-2\frac{\log 2}{\log\alpha}-
			\rho-s)}\right\}\Bigg)<\infty,
	\end{align*}
	where $\rho=\frac{\log\beta}{\log\alpha}$ and
	$C_{I_0},C_{L_1},C_{L_2}$ are positive constants.
\end{proposition}
\begin{proof}
	First, note that if $s>1$ and $0<\rho<1$, then
	\[
		\frac{1}{\rho}>\frac{s}{s-1+\rho}>1>\frac{s\rho}{s-1+\rho}>\rho.
	\]
	As stated above, we can change the order of integration in 
	\eqref{main_integral} to get
	\begin{align}
		\label{main_integral_first_estimate}
		\int\limits_{-\lambda_0}^{\lambda_0}
		\varphi_s(\mup,\hcp(v))d\lambda
		&=\int\limits_A d\muuc(\tilde x,\tilde y,\tilde z,\tilde w)
			\int\limits_{-\lambda_0}^{\lambda_0}
			\frac{d\lambda}{(\Iuc^2+(\Ic+\lambda \Ip)^2)^{s/2}}\notag\\
		&\leq C_{I_0}+
			\sum\limits_{k=k_0}^{\infty}\frac{2\lambda_0}{(1-2\alpha)^s
			\alpha^{s(k+1)}}\muuc(S_k),
	\end{align}
	with
	\[
		S_k:=\left\{(\tilde x,\tilde y,\tilde z,\tilde w)\in A: 
			\frac{2\lambda_0}{(1-2\alpha)^s \alpha^{sk}}
		<\int\limits_{-\lambda_0}^{\lambda_0}
			\frac{d\lambda}{(\Iuc^2+(\Ic+\lambda \Ip)^2)^{s/2}}
		\leq\frac{2\lambda_0}{(1-2\alpha)^s \alpha^{s(k+1)}}\right\}
	\]
	and $k_0\in\set N$ big enough to use later on the lower bound for $\Ip$.
	We have the following upper bounds for the Lebesgue integral
	\begin{equation}
		\label{Lebesgue_integral_upper_bounds}
		\int\limits_{-\lambda_0}^{\lambda_0}
			\frac{d\lambda}{(\Iuc^2+(\Ic+\lambda \Ip)^2)^{s/2}}
		\leq\begin{cases}
			\frac{2\lambda_0}{\Iuc^s} & \text{ general}\\\\
			\frac{2 s}{(s-1)\Iuc^{s-1}\abs{\Ip}} & \text{ general}\\\\
			\frac{2\lambda_0}{(\Iuc^2+(\abs{\Ic}-
				\lambda_0\abs{\Ip})^2)^{s/2}}&\text{ if } 
			\abs{\Ic}\geq(1+\eta)\lambda_0\abs{\Ip}
		\end{cases},
	\end{equation}
	where $\eta>0$.
	
	For this paragraph, we assume that $(\tilde x,\tilde y,\tilde z,\tilde w)\in S_k$.
	Using this and the first general upper bound of
	\eqref{Lebesgue_integral_upper_bounds}, we get
	\[
		\abs{y-\tilde y}\leq\Iuc<(1-2\alpha)\alpha^{k}
	\]
	and the same for $\abs{x-\tilde x}$ and $\abs{z-\tilde z}$.
	Using the second general upper bound of \eqref{Lebesgue_integral_upper_bounds} 
	and the lower bound from Lemma \ref{lemma_I_p_estimations} for $\Ip$, we get
	\[
		\abs{y-\tilde y}^{\rho}
		\leq\frac{\abs{\Ip}}{K_2}\leq C_1\frac{\alpha^{sk}}{\Iuc^{s-1}},
	\]
	with
	\[
		C_1:=\frac{s(1-2\alpha)^s}{(s-1)\lambda_0 K_2}.
	\]
	That means
	\begin{align}\label{improvement_y}
		\abs{y-\tilde y}\leq\left(C_1\alpha^{sk}\right)^{\frac{1}{s-1+\rho}}
	\end{align}
	and
	\[
		\abs{y-\tilde y}^{\frac{\rho}{s-1}}\cdot\abs{x-\tilde x}
		\leq\left(C_1\alpha^{sk}\right)^{\frac{1}{s-1}}.
	\]
	We can conclude that
	\begin{align}\label{improvement_xz}
		\abs{x-\tilde x}\leq
		\begin{cases}
			(1-2\alpha)\alpha^k &\textrm{ for } 
				0\leq\abs{y-\tilde y}\leq C_2\alpha^{\frac{k}{\rho}}\\\\
			\left(\frac{C_1\alpha^{sk}}{\abs{y-\tilde y}^\rho}\right)^{\frac{1}{s-1}}
				&\textrm{ for } C_2\alpha^{\frac{k}{\rho}}
			\leq\abs{y-\tilde y}
			\leq\left(C_1\alpha^{sk}\right)^{\frac{1}{s-1+\rho}}
		\end{cases},
	\end{align}
	with 
	\[
		C_2:=\left(\frac{C_1}{(1-2\alpha)^{s-1}}\right)^\frac{1}{\rho}.
	\]
	The same result is true for $\abs{z-\tilde z}$.
	Now, what is the upper bound for $\abs{w-\tilde w}$?
	We have
	\[
		\abs{w-\tilde w}\leq\abs{\Ic}+\abs{\Delta_g}
	\]
	and
	\[
		\abs{\Ic}\leq
		\begin{cases}
		(1+\eta)\lambda_0 K_1\abs{y-\tilde y}^{\rho}&\textrm{ for }
				\abs{\Ic}<(1+\eta)\lambda_0\abs{\Ip}\\\\
			\frac{1+\eta}{\eta}(1-2\alpha)\alpha^{k}&\textrm{ for }
				\abs{\Ic}\geq(1+\eta)\lambda_0\abs{\Ip}
		\end{cases},
	\]
	applying the third upper bound of \eqref{Lebesgue_integral_upper_bounds} and
	\[
		\abs{\Ic}-\lambda_0\abs{\Ip}\geq\abs{\Ic}-\frac{\abs{\Ic}}{1+\eta}=
			\abs{\Ic}\frac{\eta}{1+\eta}
	\]
	for the second inequality.
	Using the estimate of $\Delta_g$ from the proof of Proposition 
	\ref{proposition_conjugation}, we obtain
	\[
		\abs{w-\tilde w}\leq C_3\alpha^k + C_4 \abs{y-\tilde y}^\rho,
	\]
	with
	\[
		C_3:=\left(\frac{1+\eta}{\eta}+\frac{\Lip g}{1-\beta}\right)(1-2\alpha)
	\]
	and
	\[
		C_4:=(1+\eta)\lambda_0 K_1+\frac{\Lip g}{(1-2\alpha)^\rho}
			\left(\frac{1}{\alpha(\beta-\alpha)}+\frac{2}{(1-\beta)(1-2\alpha)}\right).
	\]
	That means
	\begin{align}\label{improvement_w}
		\abs{w-\tilde w}\leq
		\begin{cases}
			(C_3+C_4C_2^\rho)\alpha^k &\textrm{ for } 
				0\leq\abs{y-\tilde y}\leq C_2\alpha^{\frac{k}{\rho}}\\\\
			\left(\frac{C_3}{C_2^\rho}+C_4\right)\abs{y-\tilde y}^{\rho}
				&\textrm{ for } C_2\alpha^{\frac{k}{\rho}}\leq\abs{y-\tilde y}
			\leq\left(C_1\alpha^{sk}\right)^{\frac{1}{s-1+\rho}}
		\end{cases}.
	\end{align}
	
	With these upper bounds we can further estimate \eqref{main_integral_first_estimate},
	\[
		C_{I_0}+\sum\limits_{k=k_0}^{\infty}
			\frac{2\lambda_0}{(1-2\alpha)^s \alpha^{s(k+1)}}\muuc(S_k)
		\leq C_{I_0}+\sum\limits_{k=k_0}^{\infty}
			\frac{2\lambda_0}{(1-2\alpha)^s\alpha^{s(k+1)}}\muuc(B_k)
	\]
	where
	\begin{multline*}
		S_k\subset B_k:=\big\{(\tilde x,\tilde y,\tilde z,\tilde w)\in A:		
			\eqref{improvement_y}\textrm{ for }\abs{y-\tilde y},
			\eqref{improvement_xz}\textrm{ for }\abs{x-\tilde x},\abs{z-\tilde z}\\
			\textrm{ and }\eqref{improvement_w}\textrm{ for }\abs{w-\tilde w}\big\}.
	\end{multline*}
	Set $\underline\alpha_y^k:=C_2\alpha^{\frac{k}{\rho}}$,
	$\overline\alpha_y^k:=\left(C_1\alpha^{sk}\right)^{\frac{1}{s-1+\rho}}$ and
	let $\underline\alpha_x^k$, $\underline\alpha_z^k$, $\underline\alpha_w^k$ be
	the	upper bounds for $\abs{x-\tilde x}$, $\abs{z-\tilde z}$, $\abs{w-\tilde w}$
	applicable in the rage $0\leq\abs{y-\tilde y}\leq\underline\alpha_y^k$ and let
	$\overline\alpha_x^k$, $\overline\alpha_z^k$, $\overline\alpha_w^k$ be the
	upper bounds applicable in the rage
	$\underline\alpha_y^k\leq\abs{y-\tilde y}\leq\overline\alpha_y^k$.
	We estimate
	\begin{align*}
		\muuc(B_k)=\int\limits_{B_k}d\muuc(\tilde x,\tilde y,\tilde z,\tilde w)
		&\leq\int\limits_{\mathclap{B(y,\underline\alpha_y^k)}}d\nu_\alpha(\tilde y)\;
			\int\limits_{\mathclap{B(w,\underline\alpha_w^k)}}
			d\nu_\beta(\tilde w)\;
			\int\limits_{\mathclap{x-\underline\alpha_x^k}}^
			{\mathclap{x+\underline\alpha_x^k}}d\lambda(\tilde x)\;
			\int\limits_{\mathclap{z-\underline\alpha_z^k}}^
			{\mathclap{z+\underline\alpha_z^k}}d\lambda(\tilde z)\\\\
		&+\int\limits_{\mathclap{\underline\alpha_y^k\leq\abs{y-\tilde y}
			\leq\overline\alpha_y^k}}d\nu_\alpha(\tilde y)\;
			\int\limits_{\mathclap{B(w,\overline\alpha_w^k)}}d\nu_\beta(\tilde w)\;
			\int\limits_{\mathclap{x-\overline\alpha_x^k}}^
			{\mathclap{x+\overline\alpha_x^k}}d\lambda(\tilde x)\;
			\int\limits_{\mathclap{z-\overline\alpha_z^k}}^
			{\mathclap{z+\overline\alpha_z^k}}d\lambda(\tilde z).
	\end{align*}
	Hence,
	\begin{align*}
		\muuc(B_k)\leq 4\underline\alpha_x^k\underline\alpha_z^k
			\int\limits_{\mathclap{B(y,\underline\alpha_y^k)}}d\nu_\alpha(\tilde y)\;
			\int\limits_{\mathclap{B(w,\underline\alpha_w^k)}}d\nu_\beta(\tilde w)\;
		+4\int\limits_{\mathclap{\underline\alpha_y^k\leq\abs{y-\tilde y}
			\leq\overline\alpha_y^k}}\overline\alpha_x^k\overline\alpha_z^k
			d\nu_\alpha(\tilde y)\;
			\int\limits_{\mathclap{B(w,\overline\alpha_w^k)}}d\nu_\beta(\tilde w).
	\end{align*}
	Now, we have for all $y\in A_\alpha$
	\[
		\nu_\alpha(B(y,r))\leq C_\alpha r^{d_\alpha},
	\]
	where $d_\alpha:=-\frac{\log 2}{\log\alpha}$ (this is equivalently true for
	$w\in A_\beta$), see e.g.\ \cite[Theorem 3.1.1]{Barreira2008}.
	Hence,
	\begin{align*}
		\muuc(B_k)&\leq 4(1-2\alpha)^2 C_\alpha C_2^{d_\alpha}C_\beta
			(C_3+C_4C_2^\rho)^{d_\beta}\alpha^{k(2+2d_\beta)}\\
		&+4C_1^{\frac{2}{s-1}}C_\beta
			\left(\frac{C_3}{C_2^\rho}+C_4\right)^{d_\beta}\alpha^{k\frac{2s}{s-1}}
			\int\limits_{\mathclap{\underline\alpha_y^k\leq\abs{y-\tilde y}
			\leq\overline\alpha_y^k}}
			\abs{y-\tilde y}^{\rho\left(d_\beta-\frac{2}{s-1}\right)}
			d\nu_\alpha(\tilde y).
	\end{align*}
	Note that in the following
	\[
		\max\{a,b\}^e:=\max\{a^e,b^e\}.
	\]
	The last integral of the previous inequality can be estimated as follows:
	\begin{align*}
		\int\limits_{\mathclap{\underline\alpha_y^k\leq\abs{y-\tilde y}
			\leq\overline\alpha_y^k}}\abs{y-\tilde y}^
			{\rho\left(d_\beta-\frac{2}{s-1}\right)}d\nu_\alpha(\tilde y)
		&\leq 2C_\alpha\sum\limits_{i=0}^{M_k-1}
			\max\left\{\overline\alpha_y^k\alpha^{i+1},
			\overline\alpha_y^k\alpha^i\right\}^{\rho\left(d_\beta-\frac{2}{s-1}\right)}
			(\overline\alpha_y^k\alpha^i)^{d_\alpha}\\
		&\leq2C_\alpha\alpha^{-\rho\abs{d_\beta-\frac{2}{s-1}}}
			\sum\limits_{i=0}^{M_k-1}\left(\overline\alpha_y^k\alpha^i\right)^
			{\rho\left(d_\beta-\frac{2}{s-1}\right)+d_\alpha}
	\end{align*}
	and $M_k$ is determined by
	\[
		\overline\alpha_y^k\alpha^{M_k}\leq\underline\alpha_y^k
		<\overline\alpha_y^k\alpha^{M_k-1}.
	\]
	Thus,
	\[
		M_k<\left(\frac{1}{\log\alpha}
			\log\frac{C_2}{C_5}+\frac{1}{\rho}+1\right)k
	\]
	and
	\begin{align*}
		\sum\limits_{i=0}^{M_k-1}\left(\overline\alpha_y^k\alpha^i\right)^
		{\rho\left(d_\beta-\frac{2}{s-1}\right)+d_\alpha}
		&\leq\max\left\{\overline\alpha_y^k\alpha^{M_k-1},
			\overline\alpha_y^k\right\}^
			{\rho\left(d_\beta-\frac{2}{s-1}\right)+d_\alpha}M_k\\
		&\leq C_5^{\rho\left(d_\beta-\frac{2}{s-1}\right)+d_\alpha}
			\max\left\{\alpha^{\frac{k}{\rho}},\alpha^{\frac{sk}{s-1+\rho}}\right\}^
			{\rho\left(d_\beta-\frac{2}{s-1}\right)+d_\alpha}M_k,
	\end{align*}
	with
	\[
		C_5^e:=\max\left\{C_2,C_1^{\frac{1}{s-1+\rho}}\right\}^e.
	\]
	Finally, we get
	\begin{multline*}
		\frac{2\lambda_0}{(1-2\alpha)^s\alpha^{s(k+1)}}\muuc(B_k)
		\leq C_{L_1}\alpha^{k(2+2d_\beta-s)}\\
		+k\,C_{L_2}\max\left\{\alpha^{k(2+2d_\beta-s)},
			\alpha^{k\frac{s}{s-1+\rho}(3+2d_\alpha-\rho-s)}\right\},
	\end{multline*}
	with
	\[
		C_{L_1}:=\frac{8\lambda_0(1-2\alpha)^2 C_\alpha C_2^{d_\alpha}C_\beta
			(C_3+C_4C_2^\rho)^{d_\beta}}{(1-2\alpha)^s\alpha^s}
	\]
	and
	\[
		C_{L_2}:=\frac{16\lambda_0C_1^{\frac{2}{s-1}}
			C_\beta\left(\frac{C_3}{C_2^\rho}+
			C_4\right)^{d_\beta}C_\alpha\alpha^{-\rho\abs{d_\beta-\frac{2}{s-1}}}
			C_5^{\rho\left(d_\beta-\frac{2}{s-1}\right)+d_\alpha}
			\left(\frac{1}{\log\alpha}\log\frac{C_2}{C_5}+
			\frac{1}{\rho}+1\right)}{(1-2\alpha)^s\alpha^s}.
	\]
	This proves the desired inequality.
	Note that $C_{I_0}$, $C_{L_1}$ and $C_{L_2}$ are independent of $v$.
\end{proof}

For the definition of $D_L$, see Corollary \ref{corollary_D_L}.

\begin{proposition}
	\label{proposition_upper_bound_box_counting_dimension}
	For $\alpha<\beta$ and for all $g\in\Cone$ we have $\overline D_B(\Ac)\leq D_L$.
\end{proposition}
\begin{proof}
	We define for
	$(x,y,z,w)\in\Auc$ and $\varepsilon>0$
	\[
		R_{x,y,z,w,\varepsilon}:=B(x,\varepsilon)\times B(y,\varepsilon^{1/\rho})
			\times B(z,\varepsilon)\times B(w,\varepsilon).
	\]
	Using the estimate at the end of the proof of Proposition
	\ref{proposition_conjugation}, we have 
	\[
		\hc(R_{x,y,z,w,\varepsilon}\cap A)\subset B(\hc(x,y,z,w),C\varepsilon),
	\]
	with $C>0$.
	Hence,
	\begin{align*}
		\overline D_B(\Ac)&=\limsup\limits_{\varepsilon\to 0}
			\frac{\log N(\Ac,C\varepsilon)}{-\log(C\varepsilon)}\\
		&\leq\limsup\limits_{\varepsilon\to 0}\frac{\log (N([0,1),\varepsilon)
			N(A_\alpha,\varepsilon^{1/\rho})N([0,1),\varepsilon)
			N(A_\beta,\varepsilon))}{-\log(C\varepsilon)}
	\end{align*}
	because	$\Ac=\hc(\Auc)$ and $\Auc$ can be covered by finitely many 
	$R_{x,y,z,w,\varepsilon}$.
	Therefore,
	\[
		\overline D_B(\Ac)\leq 2+\frac{1}{\rho}D_B(A_\alpha)+D_B(A_\beta)
		=2-2\frac{\log 2}{\log\beta},
	\]
	where $D_B(A_\alpha)=-\frac{\log 2}{\log\alpha}$.
	This gives the desired upper bound for $\beta\leq 1/4$.

	To get the upper bound for $\beta\geq 1/4$ we define
	for $n\in\set N$ and $(l,m)\in\{0,\dots,2^{n}-1\}^2$
	\begin{align*}
		S_n^{lm}:=\left[\frac{l}{2^n},\frac{l+1}{2^n}\right)\times[0,1]\times
			\left[\frac{m}{2^n},\frac{m+1}{2^n}\right)\times
			\left[-\frac{\supnorm g}{1-\beta},1+\frac{\supnorm g}{1-\beta}\right]
	\end{align*}
	and with $R_n^{lm}\subset S_n^{lm}$ we denote a subbox of $S_n^{lm}$ 
	with the length $(\beta/2)^n$ in the $x$-direction and the same length in the
	$y$-,$z$- and $w$-direction.
	Note that each $S_n^{lm}$ can be covered by $1/\beta^n+1$ boxes of the form
	$R_n^{lm}$ and we need $2^{2n}$ boxes of the form $S_n^{lm}$ to cover the
	whole attractor $\Ac$.
	That means we can cover the attractor by $2^{2n}(1/\beta^n+1)$ boxes of the
	form $R_n^{lm}$.
	Now, consider the image of all these boxes under $\Bc^n$ and observe that
	the attractor is contained in the union of these images.
	
	Recall the notation introduced at the beginning of the proof of	Proposition
	\ref{proposition_conjugation}.
	If we consider two arbitrary points
	$(x,y,z,w),(\tilde x,\tilde y,\tilde z,\tilde w)\in R_n^{lm}$, then
	$(x)_i=(\tilde x)_i, (z)_i=(\tilde z)_i$ for $i \in\{0,\dots,n-1\}$ and
	this means
	\[
		\abs{x_n-\tilde x_n}\leq 2^n\abs{x-\tilde x}\leq\beta^n,\quad
		\abs{y_n-\tilde y_n}\leq\alpha^n,\quad
		\abs{z_n-\tilde z_n}\leq 1
	\]
	and
	\[
		\abs{w_n-\tilde w_n}\leq\beta^n\abs{w-\tilde w}+
		\Lip{g}\sum\limits_{i=0}^{n-1}\beta^{n-1-i}\left[2^i\abs{x-\tilde x}+
			\alpha^i\abs{y-\tilde y}\right]
		\leq C\beta^n\
	\]
	with
	\[
		C:=1+2\frac{\supnorm g}{1-\beta}
		+\Lip{g}\left(\frac{1}{2(1-\beta)}+\frac{1}{\beta-\alpha}\right).
	\]
	Therefore, the image of $R_n^{lm}$ under $\Bc^n$ is contained in a box
	with the lengths $\beta^n$, $\alpha^n$, $1$ and $C\beta^n$.
	Now, we cover the attractor $\Ac$ by cubes of the length $\alpha^n$ and this
	number is bounded above by
	\[
		2^{2n}\cdot\left(\frac{1}{\beta^n}+1\right)\cdot
			\left(\frac{\beta^n}{\alpha^n}+1\right)\cdot 1\cdot
			\left(\frac{1}{\alpha^n}+1\right)\cdot
			\left(\frac{C\beta^n}{\alpha^n}+1\right)
		\leq \tilde C\left(\frac{2^2\beta}{\alpha^3}\right)^n.
	\]
	Now, for $\varepsilon>0$ choose $n(\varepsilon)\in\set N$ such that
	$\alpha^{n(\varepsilon)}\leq\varepsilon/2<\alpha^{n(\varepsilon)-1}$.
	Then
	\[
		\frac{N(\Ac,\varepsilon)}{-\log\frac{\varepsilon}{2}}
		\leq\frac{\log\tilde C\left(\frac{2^2\beta}{\alpha^3}\right)^{n(\varepsilon)}}
			{\log\alpha^{n(\varepsilon)-1}}
	\]
	and hence for $\varepsilon\to 0$
	\[
		\overline D_B(\Ac)\leq\frac{\log\frac{2^2\beta}{\alpha^3}}{\log\alpha}
		=3-2\frac{\log 2}{\log\alpha}-\frac{\log\beta}{\log\alpha}.
	\]
	This gives the desired upper bound for $\beta\geq 1/4$.
	Summing-up, we get $\overline D_B(\Ac)\leq D_L$.
\end{proof}

Note that the second upper bound for the box-counting dimension in the last
proposition could be also derived by more general methods using the proofs in
\cite{Chen1993} and \cite{ChepyzhovIlyin2001}, but with the stronger requirement
that the derivative of $g$ is uniformly continuous.

Now, we have all what we need to prove the main theorem.

\begin{theorem}
	\label{prevelant_result}
	For $\alpha<\beta$ and for a prevalent set of functions $g\in\Cone$ we have
	that $\mu_g$ is exact dimensional with $d(\muc)=D_1(\muc)=D_L$.
\end{theorem}
\begin{proof}
	By using Proposition \ref{proposition_relation_pointwise_boxcounting_dimension}
	and Proposition	\ref{proposition_upper_bound_box_counting_dimension}, we get
	for all $g\in\Cone$ that $\overline d(\muc,v)\leq\overline D_B(\Ac)\leq D_L$
	for \mucae{$v$}.
	Hence, it remains to show that for a prevalent set of functions $g\in\Cone$
	we have $\underline d(\muc,v)\geq D_L$ for \mucae{$v$}.
	Recall that this is equivalent to showing that the exceptional set
	$E$ defined in \eqref{shy_set} is shy.
	Note that $E$ is a subset of
	\[
		E':=\Bigg\{g\in\Cone:\textnormal{ there is a }0\leq s<D_L
			\textnormal{ such that }\int\limits_{\Ac}
			\varphi_s(\muc,v)d\muc(v)=\infty\Bigg\}.
	\]
	We claim that Lebesgue measure supported on $P$, see
	\eqref{definition_probe}, is transverse to $E'$, i.e.\ $E'$ and therefore
	$E$ are shy.

	To show that $E'$ is a Borel set, we define the sets
	\[
		E_s':=\Bigg\{g\in\Cone:\int\limits_{\Ac}\varphi_s(\muc,v)d\muc(v)=\infty\Bigg\}.
	\]
	Note that $E_s'\subseteq E_r'$ for $0\leq s\leq r<\infty$.
	For a fixed $s\in[0,\infty)$ we will show that the map
	\begin{align}
		\label{exceptional_function}
		g\mapsto\int\limits_{\Ac}\varphi_s(\muc,v)d\muc(v)
		=\int\limits_{\Auc}\varphi_s(\muc,\hc(v))d\muuc(v):\Cone\to [0,\infty]
	\end{align}
	is lower semi-continuous, to prove that each $E_s'$ is a Borel set.
	Recall that
	\[
		\varphi_s(\mu_g,v)
		=\int\limits_{\Auc}\frac{1}{\abs{v-\hc(\tilde v)}^s}d\muuc(\tilde v)
	\]
	and observe that for a fixed $s\in[0,\infty)$ the map
	\[
		(g,v)\mapsto\varphi_s(\muc,v):\Cone\times M^2\to[0,\infty]
	\]
	is lower semi-continuous.
	This is true because for an arbitrary sequence $((g_k,v_k))_{k\in\set N}$ with
	$g_k\to g$ and $v_k\to v$ we have that $h_{g_k}\to h_g$ for $k\to\infty$, since 
	$\Abs{h_{g_k}-h_g}_{\infty}\leq C\Abs{g_k-g}_{\infty}$, $C>0$ and the $g_k$'s 
	convergence uniformly to $g$.
	Thus,
	\[
		\frac{1}{\abs{v-\hc(\tilde v)}^s}
		=\liminf\limits_{k\to\infty}\frac{1}{\abs{v_k-h_{g_k}(\tilde v)}^s},
	\]
	with $\tilde v\in\Auc$ and by using Fatou's lemma, we get
	$\varphi_s(\mu_g,v)\leq\liminf\limits_{k\to\infty}\varphi_s(\mu_{g_k},v_k)$.
	Now, by setting $v_k:=h_{g_k}(v)$ for $v\in\Auc$ and using $v_k\to\hc(v)$ for
	$k\to\infty$ and again Fatou's lemma, we get that \eqref{exceptional_function}
	is lower semi-continuous.
	Therefore, $E_s'$ is a Borel set for fixed $s\in[0,\infty)$.
	Now, consider an arbitrary sequence $(s_k)_{k\in\set N}$ with
	$s_1>1, s_k\nnearrow D_L$.
	We have
	\begin{align}
		\label{exceptional_set_countable_union}
		E'=\bigcup\limits_{0\leq s<D_L}E_s'=
		\bigcup\limits_{k\in\set N}E_{s_k}',
	\end{align}
	i.e.\ $E'$ is a Borel set, too.
	
	Next, using Tonelli's theorem and Proposition \ref{main_proposition},
	\begin{align}
		\label{s_potential_method_Hausdorff}
		\int\limits_{-\lambda_0}^{\lambda_0}\;\int\limits_{\Ac}
			\varphi_s(\mup,v)d\mup(v)d\lambda
		=\int\limits_{\Auc}\;\int\limits_{-\lambda_0}^{\lambda_0}
			\varphi_s(\mup,\hcp(v))d\lambda d\muuc(v)
		<\infty,
	\end{align}
	for all $1<s<D_L$.
	This tells us that the intersection of $E_s'-g$ with a line segment in $P$
	of length $2\lambda_0$ has measure $0$ with respect to Lebesgue measure on $P$.
	Since $\lambda_0 > 0$ is arbitrary, by taking a countable union we conclude
	that the intersection of $E_s'-g$ with $P$ has measure $0$.
	Then since $g$ is arbitrary, Lebesgue measure on $P$ is transverse to $E_s'$.
	Furthermore, from \eqref{exceptional_set_countable_union} it follows that
	the intersection of $E'-g$ with $P$ has measure $0$, and Lebesgue measure
	on $P$ is transverse to $E'$ as	claimed.	
	This means $E$ is shy and this establishes for a prevalent set of $g$'s the
	lower bound for the pointwise dimension.
	Hence, for a prevalent set of functions $g\in\Cone$ we have $d(\muc)=D_L$.
\end{proof}

\section{Discussion}

We demonstrated that the potential theoretic method for the pointwise dimension,
together with the notion of prevalence, could provide a useful approach to tackle
families of systems with physical measures $\mu$ for which the Kaplan-Yorke
equality $D_1(\mu)=D_L(\mu)$ is violated for at least one member of the family.
Our proofs are for system {\eqref{coupled_skinny_bakers_maps} with $f=0$, i.e.,
uni-directional coupling.
Note that the differentiability of the coupling function $g$ in system 
\eqref{coupled_skinny_bakers_maps} is actually only needed for the definition of
the Lyapunov exponents (in all the proofs, excluding Lemma
\ref{lemma_lyapunov_exponents}, we only need that $g$ is bounded and Lipschitz
continuous).
Also note that the uncoupled ($g=0$) and coupled system ($g\neq 0$) have the
same topological and measure-theoretic entropy, namely, $2\log 2$.
Further,  note that in the case of the prevalent result ($\alpha<\beta$) the
exceptional set of coupling functions, i.e.\ the subset of $\Cone$ where the
equality of the information and Lyapunov dimension is violated, is nontrivial.
For example, consider any Lipschitz continuous function $\tilde g:M\to\set R$
such that $g:=\tilde g\circ B_{\alpha}-\beta\cdot\tilde g\in\Cone$, e.g.\
$\tilde g(x,y)=\sin^2(2\pi x)\tanh(y)$.
Then the conjugacy $\hc$ (see Proposition \ref{proposition_conjugation}) from
the uncoupled to the coupled system has in the $w$-component the form
$w+\tilde g(x,y)$, and we get for the information dimension of the coupled
system $D_1(\muc)=D_1(\muuc)< D_L$, where $D_L=D_L(\muuc)=D_L(\muc)$ (see
Corollary \ref{corollary_D_L}).
We do not know whether there exist coupling functions $g$ such that 
$D_1(\muuc)<D_1(\muc)< D_L$, nor are we able to completely classify the 
exceptional cases (in \cite{KaplanMallet-ParetYorke1984} they were able to do
this for coupling functions on the torus with sufficient smoothness).

As already mentioned, by showing that the physical measure $\muc$ is exact
dimensional with $d(\muc)=D_L$ (for $\alpha<\beta$ and a prevalent set of $g$'s)
in Theorem \ref{prevelant_result}, we get that several other dimensions of
$\muc$ coincide with the Lyapunov dimension, too.
This supports the formulation of the Kaplan-Yorke conjecture that can be found in 
\cite[Conjecture 2]{FarmerOttYorke1983}.
Furthermore, using the potential theoretic method for the Hausdorff dimension
\cite[Theorem 4.13]{Falconer2003} and \eqref{s_potential_method_Hausdorff}
together with Proposition \ref{proposition_upper_bound_box_counting_dimension},
we get that the Hausdorff dimension and the box-counting dimension of the
attractor coincide with the Lyapunov dimension in a prevalent sense.
This equality is supposed to be a rather rare phenomenon, according to Conjecture
3 in \cite{FarmerOttYorke1983}, occurring only when every point on the attractor
yields the same Lyapunov exponents.
More commonly, the support  of a physical measure contains points (for example,
unstable periodic orbits) whose Lyapunov exponents are different from that of the
physical measure.
In the system we study, every point for which the Lyapunov exponents exist has
the same exponents as the physical measure, though there is a null set where the
Lyapunov exponents are not defined because of the discontinuity. 
From \eqref{s_potential_method_Hausdorff} we can also conclude the following for
the dimension spectrum of the physical measure, where we use the integral based
definition which is given for a measure $\mu$ by
\[
	D_q(\mu):=\lim\limits_{\varepsilon\to 0}\frac{\log\int
		\mu(B(v,\varepsilon))^{q-1}d\mu(v)}{(q-1)\log\varepsilon},
\]
for $q\in\set R$ and $q\neq 1$, see \cite{HentschelProcaccia1983} (if the limit
does not exist, consider the $\limsup$ and $\liminf$, respectively).
$D_2(\mu)$ is called the correlation dimension of $\mu$.
There is also a potential theoretic method for the correlation dimension (see
\cite[Proposition 2.3]{SauerYorke1997}) and by using 
\eqref{s_potential_method_Hausdorff} again, we get $D_2(\muc)\geq D_L$ in a 
prevalent sense.
Since $D_q(\mu)$ is a non-increasing function of $q$, we get $D_q(\muc)=D_L$ for
$0\leq q\leq 2$ in a prevalent sense.
It is an interesting question whether the whole spectrum equals $D_L$, i.e.,
whether or not the coupling function affects the spectrum and if so would it be
possible to draw conclusions about the coupling function using the spectrum (for
$\alpha\geq\beta$ we have in the uncoupled and coupled case a monofractal).

Our results are related to literature on filtering of chaotic signals, starting
with the observation \cite{BadiiPoliti1986}, \cite{Badiietal1988},
\cite{MitschkeMoellerLange1988} that applying a linear filter to a chaotic signal
could increase the Lyapunov dimension (and presumably other dimensions, by the
Kaplan-Yorke conjecture) of the attractor reconstructed from the signal.
The underlying scenario is that of uni-directional coupling from a chaotic
subsystem (which produces a signal, and which we call the drive system below) to
a contracting subsystem (the filter).
If the drive subsystem is invertible, an attractor of the coupled system can be
thought of as a graph over an attractor of the drive subsystem.
Dimension increase (of the coupled system versus the drive subsystem) can arise
if the graph is nonsmooth, but under appropriate hypotheses the graph is
Lipschitz if the Lyapunov exponents of the filter are all smaller than the
Lyapunov exponents of the drive subsystem \cite{BroomheadHukeMuldoon1992},
\cite{StarkDavies1994}, \cite{PecoraCarroll1996}, \cite{DaviesCampbell1996},
\cite{Stark1997}, \cite{Stark1999}.
In the latter case, the Lyapunov dimension of an attractor of the coupled system
is the same as the Lyapunov dimension of the corresponding attractor of the drive
subsystem, and the information dimensions of the two attractors are the same
too, so the filter does not affect the Kaplan-Yorke conjecture.
On the other hand, the results of \cite{KaplanMallet-ParetYorke1984} show that
the conjecture also holds in a particular scenario where the graph is
non-Lipschitz.

A fundamental difference between our scenario and the filtering scenario is
that because we couple two chaotic subsystems, the Lyapunov and information
dimensions are already different in our uncoupled system.
This inequality persists when the contraction of the drive subsystem is weaker
than the contraction of the subsystem it drives (which is analagous to the case
where the graph is smooth in the filter scenario, except in that case 
\emph{equality} of dimensions persists).

Next, we discuss how the relative size of the contraction rates in our two
subsystems affects the geometry of the attractor of the coupled system.
Recall that we are considering the system \eqref{coupled_skinny_bakers_maps}
with $f=0$, making the $x$-$y$ baker's map the drive system and the
$z$-$w$ baker's map the driven (or ``response'') system.
Above, we compared the cases $\alpha>\beta$ and $\alpha<\beta$ to two cases for
a contracting response system (filter), namely, that the attractor of the
coupled system is a Lipschitz graph versus a non-Lipschitz graph, but the
geometry is fundamentally different in our scenario.
Algebraically, we have characterized the difference as follows.
For $\alpha>\beta$, by Proposition \ref{proposition_conjugation} 
there exists a bi-Lipschitz conjugacy between the attractor of the uncoupled
system and the attractor of the coupled system which implies that the physical
measures on the two attractors have the same information dimension.
For $\alpha<\beta$ and a prevalent set of coupling functions, by Theorem
\ref{prevelant_result} the information dimension of the physical measure is
strictly greater for the coupled system than for the uncoupled system, which
denies the existence of a bi-Lipschitz conjugacy.
Geometrically, we can interpret  the difference as follows, considering a
$w$-$y$ cross-section of the attractor as in Figure \ref{sample_attractor}.
With no coupling ($f=g=0$), the cross-section is the Cartesian product of
two Cantor sets; coupling with nonzero $g$ shears this product.
When $\alpha>\beta$, we argue below that the amount of shear is uniformly
bounded at all scales, whereas for $\alpha<\beta$, the shear can grow even
stronger at smaller scales.
This is related to the fact that the strongly and weakly stable directions at a
given point on the attractor are as shown in Figure \ref{stable_directions};
in particular, the $w$-direction is the strongly stable direction for
$\alpha>\beta$ and the weakly stable direction for $\alpha<\beta$ (see the proof
of Lemma \ref{lemma_lyapunov_exponents} and the comment that follows).
The local dynamics consists of the composition of a shear and a contraction,
both of which preserve the $w$-direction.
The amount of shear is position-dependent but is bounded in a single iteration.
When $\alpha>\beta$, the contraction reduces the shear (due to the stronger
contraction in the $w$-direction), and implies that the cumulative amount of
shear over multiple iterations remains bounded.
When $\alpha<\beta$, the contraction amplifies the shear and allows the
cumulative shear to become unbounded as the number of iterations increases.

Now, coming back to the question that we posed in the introduction: if the
uncoupled system violates the Kaplan-Yorke equality, does coupling typically 
restore equality?
In the case of uni-directional coupling, the answer depends on the direction of
the coupling (but not the size), according to our main result.
We also conjectured in the introduction that for system 
\eqref{coupled_skinny_bakers_maps} the Kaplan-Yorke equality holds in a
prevalent sense in the case of bi-directional coupling.
This would answer the question we just posed in the positive, i.e., in the case
of bi-directional coupling the Kaplan-Yorke equality would prevalently hold even
if the uncoupled system violates it.
We remark that two proximate physical systems are likely to
have at least a small bi-directional coupling, even if the primary interaction is
one-way.
On the other hand, if the coupling in a given direction is small enough compared
to the scales of interest, then at these scales the dynamics and their
dimensionality may be indistinguishable from the case of zero coupling.

\begin{figure}[t]
	\includegraphics[width=1.0\textwidth]{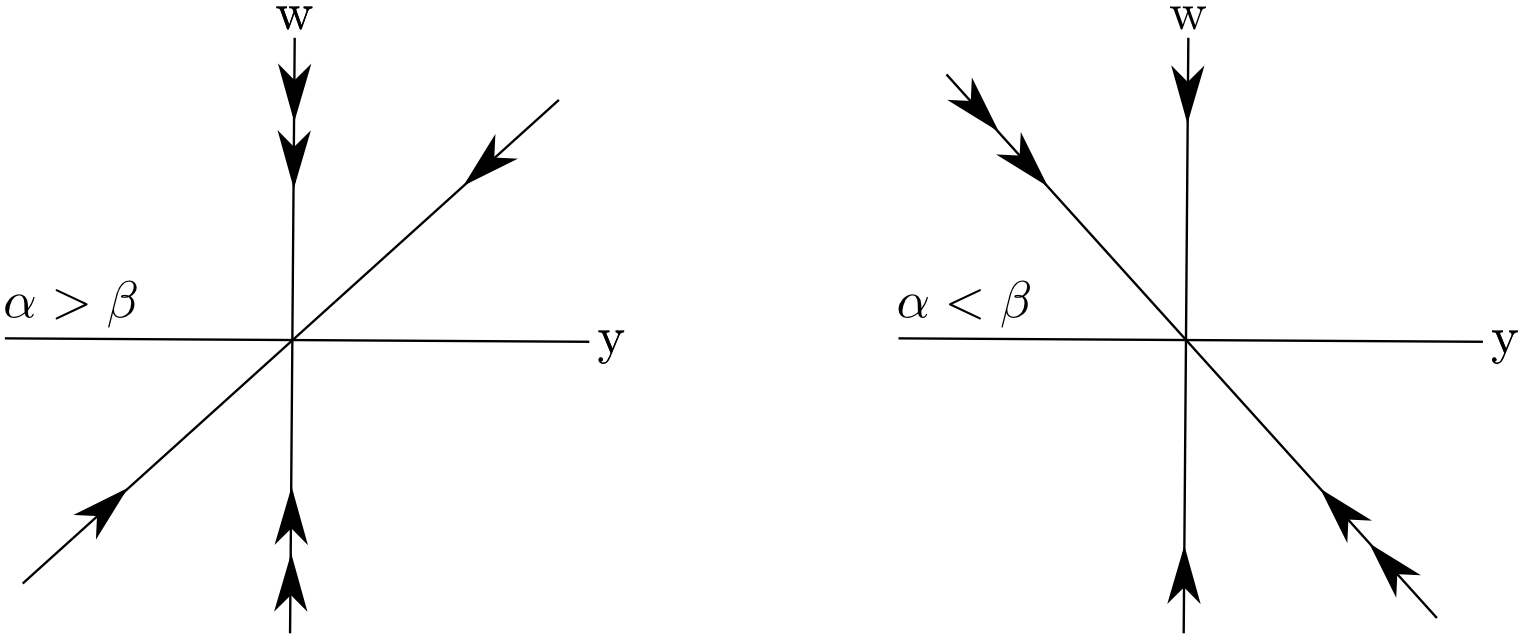}
	\caption{Sketch of the strongly and weakly stable directions at a given point
		for system \eqref{coupled_skinny_bakers_maps} with $f=0$ for the cases
		$\alpha>\beta$ and $\alpha<\beta$.}
	\label{stable_directions}
\end{figure}

The results of this article suggest that for uni-directionally coupled
(skew-product) systems where the Kaplan-Yorke equality is robustly violated there
should be still a meaningful relation between the information dimension of the
physical measure and the Lyapunov exponents.
However, the relation must distinguish between the Lyapunov exponents of the
drive system (the ``base'', in the language of skew-product systems) and those
associated with the response system (the ``fiber'').
To be precise, let $\mu$ be a physical measure for the combined drive-response
system.
The projection of $\mu$ onto the drive state space is invariant for the drive
system, and its Lyapunov exponents $\lambda_1 \geq \cdots \geq \lambda_k$ are
also Lyapunov exponents of the combined system. 
Let $\chi_1 \geq \cdots \geq \chi_\ell$ be the remaining Lyapunov exponents of
the combined system.
We conjecture that there is a function of the two vectors
$(\lambda_1,\dots,\lambda_k)$ and $(\chi_1,\dots,\chi_\ell)$ that typically
coincides with $D_1(\mu)$, and in this sense extends the Kaplan-Yorke conjecture
to uni-directionally coupled systems.
Our results indicate that depending on the relative ordering of the $\lambda_i$'s
and the $\chi_j$'s, this function may coincide with the Lyapunov dimension of the
combined set of exponents, or may coincide with the sum of the Lyapunov dimensions
of the two sets of exponents considered separately.
There may be other possibilities as well in higher dimensions.
More generally, one can also consider systems with more than two subsystems and
incomplete (not all-to-all) coupling.
Also in this context, a very interesting question would be whether certain
dimensions could be a measure for the connectivity of a network of coupled
systems.

It seems that prevalence could be a good notion to define what ``typical'' means
in the Kaplan-Yorke conjecture.
But in the general case of arbitrary maps on manifolds we have no natural linear
structure on the function space and nonlinear notions of prevalence are
still subject of ongoing investigations, see \cite{HuntKaloshin2010}.

\bibliographystyle{alpha}
\bibliography{lit}

\end{document}